\documentclass[11pt]{amsart}
\usepackage[a4paper, top=3cm, bottom=3cm]{geometry}
\usepackage{amsmath}
\usepackage{extarrows}
\usepackage{amssymb,amsfonts}
\usepackage[all,arc]{xy}
\usepackage{enumerate}
\usepackage{enumitem}
\usepackage{mathrsfs}
\usepackage[hyperindex]{hyperref}
\usepackage{color}
\usepackage{mathtools}
\usepackage[utf8]{inputenc}
\usepackage{psfrag,euscript}
\usepackage{graphicx}
\usepackage{tikz}
\usetikzlibrary{arrows}
\usepackage{verbatim}
\usepackage{tikz-cd}
\usetikzlibrary{3d,positioning}
\usepackage{rotating}
\usepackage{tabularx}
\usepackage{multirow}
\usepackage{arydshln}
\usepackage{multicol}
\usepackage{babel} 
\theoremstyle{definition}

\newtheorem{teorema}{Theorem}[section]

\newtheorem{proposicao}[teorema]{Proposition}

\newcommand{\downarrowtail}{\vcenter{\hbox{\rotatebox{90}{$\leftarrowtail$}}}}

\theoremstyle{definition}
\newtheorem{definicao}[teorema]{Definition}

\newtheorem{exemplo}[teorema]{Example}
\newtheorem{examples}[teorema]{Examples}

\newtheorem{obs}[teorema]{Remark}

\theoremstyle{remark}

\usepackage[nodisplayskipstretch]{setspace}
\newcommand{\vertbold}{\;\vcenter{\hbox{\rule{1.2pt}{1.2em}}}\;}

\makeatletter

\title[Manin triples for double Lie bialgebroids]{{\large M}anin triples for double {\large L}ie bialgebroids}

\author{{\small A}na {\small C}arolina {\small M}ançur \\ 
{\small \href{mailto:carol.mancu@gmail.com}{\texttt{\lowercase{carol.mancu@gmail.com}}}}
}

\begin{document}

\begin{abstract} 

We verify that LA-Courant algebroids provide the Manin triple framework for double Lie bialgebroids. Specifically, we establish a correspondence between double Lie bialgebroids and LA-Manin triples, i.e., LA-Courant algebroids equipped with a pair of complementary LA-Dirac structures. As an application, LA-Courant algebroids and CA-groupoids given by Drinfeld doubles are shown to correspond via integration and differentiation.

\end{abstract}

\maketitle
\pagestyle{plain}

\tableofcontents

\section{Introduction}

Lie bialgebras arise as the infinitesimal versions of Poisson Lie groups \cite{drinfel1990hamiltonian} and have an important characterization as Manin triples \cite{drinfel1988quantum}, i.e., quadratic Lie algebras equipped with a pair of transverse Lagrangian subalgebras. This description was generalized to Lie bialgebroids by Liu, Weinstein, and Xu \cite{liu1997manin}, who demonstrated that Courant algebroids provide the Manin triple framework for Lie bialgebroids: there is a one-to-one correspondence between Lie bialgebroids and Courant algebroids with two transverse Dirac structures. In this paper, we extend this correspondence to the setting of ``double structures over Lie algebroids", as explained below.

\vspace{1mm}

Poisson double groupoids \cite{mackenzie1999symplectic} have two differentiation steps, leading to ``double bialgebroids" at two levels, namely, Lie bialgebroid groupoids and double Lie bialgebroids \cite{bursztyn2021poisson}. In \cite{mancur1}, the author showed that Lie bialgebroid groupoids, given by LA-groupoids in duality, correspond to multiplicative Manin triples in {\em CA-groupoids} \cite{li2012courant, cristianthesis, mehta2009q} (or {\em multiplicative Courant algebroids}). In this paper, we aim to take a step further by proving the infinitesimal version of this correspondence, i.e., that LA-Courant algebroids provide the Manin triple framework for double Lie bialgebroids \cite{bursztyn2021poisson}. Since double Lie algebroids and LA-Courant algebroids are more intricate than LA-groupoids and CA-groupoids, respectively, establishing the corresponding result in the infinitesimal setting involves additional technical difficulties.

\vspace{1mm}

Double Lie bialgebroids \cite{bursztyn2021poisson} consist of two double Lie algebroids in duality, as shown below,
$$
\begin{tikzcd}[
		column sep={2em,between origins},
		row sep={1.3em,between origins},]
		\Omega & \Longrightarrow & H  \\
		\Downarrow & & \Downarrow \\
		A & \Longrightarrow & M,
\end{tikzcd}
\quad
\begin{tikzcd}[
		column sep={2em,between origins},
		row sep={1.3em,between origins},]
		\,\, \Omega^{*_A} \,\,\, & \Longrightarrow & \,\, C^*  \\
		\Downarrow \,\,\, & & \Downarrow \\
		A \,\,\, & \Longrightarrow & M,
\end{tikzcd}
$$
such that the Lie algebroids $\Omega \Rightarrow A$ and $\Omega^{*_A} \Rightarrow A$ form a Lie bialgebroid, with $\Rightarrow$ indicating a Lie algebroid structure. To describe double Lie bialgebroids via Manin triples, the key point is identifying the double structure of their Drinfeld double. In the setting described above,
$\Omega \oplus_A \Omega^{*_A} \rightarrowtail A$ is a Courant algebroid\footnote{Throughout the paper, the arrow $\rightarrowtail$ denotes a Courant algebroid structure.} which carries an additional Lie algebroid structure $\Omega \oplus_A \Omega^{*_A} \Rightarrow H \oplus C^*$. Our central result is that these two structures are compatible in the sense that they form an {\em LA-Courant algebroid} \cite{li2012courant}. This leads to Theorem \ref{teorema2}, where we prove the correspondence between double Lie bialgebroids and LA-Courant algebroids equipped with two complementary LA-Dirac structures.

\vspace{1mm}

The core of the technical work in this paper is to show how the compatibility conditions of a double Lie bialgebroid, after taking the Drinfeld double, become precisely those of an LA-Courant algebroid. Describing LA-Courant algebroids in classical terms involves intricate compatibility conditions, due to their formulation in terms of triple vector bundles. To define such structure, Li-Bland \cite{li2012courant} considered the linear Poisson structure on the dual of the Lie algebroid, viewed as a morphism of triple vector bundles, and required a certain relation associated with this map to be a Courant relation. The key step in our approach is realizing that such a relation may be considered even when starting with only a Lie algebroid, as observed in \cite{jotz2020courant}. We then progressively include additional structures in a compatible manner and observe how this relation behaves, until we characterize the compatibility conditions of double Lie algebroids through it (see \S \ref{subsectiontherelation}). In the following diagram, we illustrate how our result completes the big picture. 

\vspace{1mm}

{\footnotesize
\bgroup
\def\arraystretch{1.6}
\begin{center}
\begin{tabular}{ |c|c|c| }
\hline
{\bf \small Global object} & {\bf \small Infinitesimal object}  & {\bf \small Drinfeld double} \\ \hline
Poisson Lie group & Lie bialgebra & quadratic Lie algebra
\\ \hline
Poisson groupoid & Lie bialgebroid & Courant algebroid
\\ \hline
Poisson double groupoid & Lie bialgebroid groupoid  & CA-groupoid
\\ \cdashline{2-3}
 & double Lie bialgebroid  & LA-Courant algebroid
\\ \hline
\end{tabular}
\end{center}
\egroup
}

\vspace{2mm}

In \cite{li2012courant}, Li-Bland asserts, through the use of graded manifolds, that LA-Courant algebroids are the infinitesimal counterpart of CA-groupoids (see also \cite[\S 7.3]{ten2019applications}). However, a direct argument for this relation is lacking.
In view of the integration results by Bursztyn, Cabrera, and del Hoyo \cite{bursztyn2021poisson}, our main result establishes a correspondence, via integration and differentiation, between LA-Courant algebroids and CA-groupoids given by Drinfeld doubles.

\vspace{1mm}

The paper is organized as follows. Section \S \ref{section2} recalls concepts and results on Manin triples for Lie bialgebroids. Section \S \ref{section3} reviews the notions of double structures over Lie algebroids, in particular, the definition of double Lie bialgebroid. In Section \S \ref{section4} we carry out a detailed study of the relation involved in the definition of LA-Courant algebroids and prove the main result of this paper (Theorem \ref{teorema2}). Section \S \ref{section5} illustrates an application of Theorem \ref{teorema2} by proving an integration/differentiation result between LA-Courant algebroids and CA-groupoids given by Drinfeld doubles.

\vspace{1mm}

\subsection*{Notations} We use the notation $G \rightrightarrows M$ for Lie groupoids, $A \Rightarrow M$ for Lie algebroids, and $E \rightarrowtail M$ for Courant algebroids.

\subsection*{Acknowledgments} The author thanks H. Bursztyn for his continuous support, insightful discussions and valuable suggestions; J. P. Ayala for the helpful conversations; D. Álvarez and M. Cueca for useful comments on a preliminary version of this note; and the anonymous referee for the careful reading and constructive recommendations. This work was partially supported by CAPES (Coordination for the Improvement of Higher Education Personnel), IMPA (Institute for Pure and Applied Mathematics), and by FAPESP (São Paulo Research Foundation) under grant \#2025/10946-6.

\section{Review of Manin triples for Lie bialgebroids}\label{section2}

In this section, we revisit definitions and results on the Manin triples theory for Lie bialgebroids.

\begin{definicao}
	Let $A \Rightarrow M$ be a Lie algebroid and suppose its dual bundle $A^* \rightarrow M$ also has a Lie algebroid structure. We say that $(A, A^*)$ is a \textbf{Lie bialgebroid} if the Lie algebroid differential $d_A : \Gamma(A^*) \rightarrow \Gamma(\wedge^2 A^*)$ is a derivation of the Schouten bracket $[\, , \,]_{A^*}$ on $\Gamma(A^*)$ in the sense that
	$$
	d_{A}[\alpha, \beta]_{A^*} = [d_A\alpha, \beta]_{A^*} + [\alpha, d_A\beta]_{A^*},
	$$
	for all $\alpha, \beta \in \Gamma(A^*)$. 
\end{definicao}

The well-known duality between linear Poisson structures and Lie algebroids establishes a natural bijection between Lie algebroid structures on $A \rightarrow M$ and linear Poisson structures on $A^* \rightarrow M$. There are different ways to describe the linearity of a Poisson structure (see \cite[\S 10.3]{mackenzie2005general}); in this paper, we use the following: given a vector bundle $E \rightarrow M$, a Poisson structure $\pi \in \mathfrak{X}^2(E)$ is \textbf{linear} if $\pi^{\sharp}: (T^*E \rightarrow E^*) \rightarrow (TE \rightarrow TM)$ is a vector bundle morphism. In this case, we say that $(E \rightarrow M,\pi)$ is a \textbf{Poisson vector bundle}. Recall that a Lie algebroid $A \Rightarrow M$ gives rise to tangent and cotangent Lie algebroids, $TA \Rightarrow TM$ and $T^*A \Rightarrow A^*$. If $\pi \in \mathfrak{X}^2(A)$ is a Poisson structure on a Lie algebroid $A \Rightarrow M$, we say that the pair $(A \Rightarrow M, \pi)$ is a \textbf{Poisson algebroid} if the map $\pi^{\sharp} : T^*A \rightarrow TA$ is a Lie algebroid morphism. Mackenzie and Xu \cite{mackenzie1994lie} showed that $(A \Rightarrow M, A^* \Rightarrow M)$ forms a Lie bialgebroid if and only if $(\pi, A \Rightarrow M)$ is a Poisson algebroid.

\begin{exemplo}\label{exemplosbialgebroids}
	Given $M$ a smooth manifold, its cotangent bundle $T^*M$ with the canonical symplectic form $\omega_{can}$ is a Poisson vector bundle whose dual Lie algebroid is $TM \Rightarrow M$. For a Poisson structure $\pi \in \mathfrak{X}^2(M)$, we can consider its associated cotangent Lie algebroid, denoted by $T^*M_{\pi}$. This construction gives rise to a linear Poisson structure on $TM$, coinciding with the tangent lift of $\pi$, denoted by $\pi^{tan}$. The pair $(TM, T^*M_{\pi})$ forms a Lie bialgebroid, and both $(T^*M_{\pi}, \omega_{can})$ and $(TM, \pi^{tan})$ are Poisson algebroids.
\end{exemplo}

A closely related notion to Lie bialgebroids is that of Courant algebroids \cite{liu1997manin}, which we recall using the non-skew-symmetric bracket formulation \cite{dorfman1993dirac, roytenberg1999courant}:

\begin{definicao}
	A \textbf{Courant algebroid} over a manifold $M$ is a vector bundle $\mathbb{E} \rightarrow M$, together with a bundle map $\rho : \mathbb{E} \rightarrow TM$, called \textit{anchor}, a non-degenerate symmetric bilinear form $\langle \,\, , \, \rangle$, and a bracket $[\![ \,\, , \, ]\!]$ on $\Gamma(\mathbb{E})$, satisfying the following axioms:
	\begin{enumerate}[left=10pt, itemsep=0.3em]
		\item[(C$1$)] $[\![ e_1 , [\![ e_2 , e_3 ]\!] ]\!] = [\![ [\![ e_1 , e_2 ]\!] , e_3 ]\!] + [\![ e_2 , [\![ e_1 , e_3 ]\!] ]\!]$;
		\item[(C$2$)] $\rho(e_1) \langle e_2 , e_3 \rangle = \langle [\![ e_1 , e_2 ]\!] , e_3 \rangle + \langle e_2 , [\![ e_1 , e_3 ]\!] \rangle$;
	    \item[(C$3$)] $[\![ e_1 , e_2 ]\!] + [\![ e_2 , e_1 ]\!] = 2 \rho^*(d \langle e_1, e_2 \rangle)$;
	    \item[(C$4$)] $[\![ e_1 , f e_2 ]\!] = f [\![ e_1 , e_2 ]\!] + \rho(e_1)(f)e_2$;
	\end{enumerate}
for $e_1, e_2, e_3 \in \Gamma(\mathbb{E})$ and $f \in C^{\infty}(M)$, where $\rho^* : T^*M \rightarrow \mathbb{E}^* \simeq \mathbb{E}$ is the dual map to $a$. We denote by $\overline{\mathbb{E}}$ the Courant algebroid with the same anchor and bracket of $\mathbb{E}$, but pairing $- \langle \,\, , \, \rangle$.
\end{definicao}

\begin{exemplo}\label{standardCA}
	The \textit{standard Courant algebroid} over a manifold $M$ is $\mathbb{T}M := TM \oplus T^*M$ with anchor map given by the projection on the first factor, bilinear form given by $\langle (X, \alpha), (Y, \beta) \rangle := \frac{1}{2} (\beta(X) + \alpha(Y))$, and bracket $[\![ (X, \alpha) , (Y, \beta) ]\!] := ([X, Y], \mathcal{L}_X\beta - \iota_Yd\alpha)$, for $X, Y \in \mathfrak{X}(M)$ and $\alpha, \beta \in \Omega^1(M)$.
\end{exemplo}

\vspace{1mm}

Let $\mathbb{E} \rightarrowtail M$ be a Courant algebroid. For any vector subbundle $L\to S$ of $\mathbb{E}\to M$, we can consider the \textbf{orthogonal complement} $L^{\perp}$ with respect to the induced metric, defining a new subbundle of $\mathbb{E}$ supported on $S$. Recall that there is a canonical identification $(L^{\perp})^{\perp} = L$. We will denote by $\Gamma(\mathbb{E}; L)$ the sections of $\mathbb{E}$ which take values in the subbundle $L$ when restricted to $S$, i.e.
$$
\Gamma(\mathbb{E}; L) := \{ e \in \Gamma(\mathbb{E}) \,\, \vert \,\, e\vert_S \in \Gamma(L) \}.
$$

\vspace{1mm}

\begin{definicao}
	Let $\mathbb{E} \rightarrowtail M$ be a Courant algebroid with anchor $\rho$ and bracket $[\![ \,\, , \, ]\!]$. A \textbf{Dirac structure with support on} $S$ is a subbundle $L \rightarrow S$ over a submanifold $S \subseteq M$, such that 
	\begin{enumerate}[left=10pt, itemsep=0.3em]
		\item[(1)] $L\vert_s$ is Lagrangian in $\mathbb{E}\vert_s$ for all $s \in S$, i.e. $L|_{s}=L|_{s}^{\perp}$;
		\item[(2)] $L$ is involutive, i.e. if for any $e_1, e_2 \in \Gamma(\mathbb{E};L)$ we have $[\![ e_1 , e_2 ]\!] \in \Gamma(\mathbb{E}; L)$;
		\item[(3)] $L$ is compatible with the anchor, i.e. $\rho(L) \subseteq TS$.
	\end{enumerate}
When $S = M$, $L$ is simply called a \textbf{Dirac structure}, and we say that $(\mathbb{E}, L)$ is a \textbf{Manin pair}. A triple $(\mathbb{E}, L_1, L_2)$ where $\mathbb{E}$ is a Courant algebroid and $L_1, L_2$ are transversal Dirac structures (i.e. $\mathbb{E} = L_1 \oplus L_2$) is called a \textbf{Manin triple}.
\end{definicao}

\begin{obs}
	Given $\mathbb{E} \rightarrowtail M$ a Courant algebroid and $L \rightarrow M$ a Dirac structure fully supported, then $(L, [\![ \,\, , \, ]\!]\vert_{\Gamma(L)}, \rho\vert_L)$ is a Lie algebroid. 
\end{obs}

\begin{exemplo}\label{exgraphdirac}
	In Example \ref{standardCA}, both $TM$ and $T^*M$ are Dirac structures in $\mathbb{T}M$.
\end{exemplo}

\begin{definicao}
    If $E_1 \to M_1$ and $E_2 \to M_2$ are vector bundles, a \textbf{VB-relation} $R : E_1 \dashrightarrow E_2$ is a subbundle $R \subseteq E_2 \times E_1$ along a submanifold $S \subseteq M_2 \times M_1$. Given Courant algebroids $\mathbb{E}_1\rightarrowtail M_1$ and $\mathbb{E}_2 \rightarrowtail M_2$, a \textbf{Courant relation} $R : \mathbb{E}_1 \dashrightarrow \mathbb{E}_2$ is a Dirac structure $R \subseteq \mathbb{E}_2 \times \overline{\mathbb{E}}_1$ along a submanifold $S \subseteq M_2 \times M_1$. If $S$ is the graph of a map $M_1 \rightarrow M_2$, then $R$ is called a \textbf{Courant morphism}.
\end{definicao}

\vspace{1mm}

 Given a VB-relation $R : E_1 \dashrightarrow E_2$, we let $ann^\natural(R) : E_2^* \dashrightarrow E_1^*$ be the relation such that $(\mu_2, \mu_1) \in ann^\natural(R)$ if $\langle \mu_1, e_1 \rangle = \langle \mu_2, e_2 \rangle$ whenever $(e_2, e_1) \in R$. Thus, $(\mu_2, \mu_1) \in ann^\natural(R)$ if and only if $(\mu_2, -\mu_1) \in ann(R)$.

\begin{exemplo}\label{standardCR}
	If $S : M_1 \dashrightarrow M_2$ is a relation, then $R_S := TS \oplus ann^{\natural}(TS) \subseteq \mathbb{T}M_2 \times \overline{\mathbb{T}M_1}$ defines a Courant relation $ R_S : \mathbb{T}M_1 \dashrightarrow \mathbb{T}M_2$ (see \cite{li2012courant}).
\end{exemplo}

\vspace{1mm}

The relation between Courant algebroids and Lie bialgebroids is established by the following result from Liu-Weinstein-Xu \cite{liu1997manin}:
\begin{itemize}[left=10pt, topsep=3pt, itemsep=0.3em]
\item Given a Lie bialgebroid $(A,A^*)$, the direct sum $A\oplus A^*$ has a natural Courant algebroid structure such that $(A\oplus A^*, A, A^*)$ is a Manin triple;
\item Given a Manin triple $(\mathbb{E}, L_1, L_2)$, the Lie algebroids $L_1$ and $L_2$ form a Lie bialgebroid, where $L_2\simeq L_1^*$ via the pairing.
\end{itemize}

\vspace{2mm}

The next result extends \cite[Proposition 7.1]{liu1997manin}, with a similar proof.

\begin{proposicao}\label{proposicao}
{\em 	Let $(A \Rightarrow M, A^* \Rightarrow M)$ be a Lie bialgebroid. Then, the vector bundles $B \rightarrow N$ and $ann(B) \rightarrow N$ are Lie subalgebroids of $A \Rightarrow M$ and $A^* \Rightarrow M$, respectively, if and only if, $L := B \oplus ann(B) \rightarrow N$ is a Dirac structure on $A \oplus A^*$ with support on $N$.}
\end{proposicao}

\section{Lie bialgebroids over Lie algebroids}\label{section3}

In order to extend the Manin triple framework to double Lie bialgebroids, we must replace Lie bialgebroids with double Lie algebroids \cite{mackenzie2006ehresmann}, as briefly recalled below.

\vspace{1mm}

A (horizontal) \textbf{VB-algebroid} \cite{mackenzie9808081double} is a double vector bundle whose (horizontal) vector bundles are equipped with a Lie algebroid structure, 
\begin{eqnarray} \label{defvbalgebroid}
	\begin{tikzcd}[
		column sep={2em,between origins},
		row sep={1.3em,between origins},]
		\Omega & \Longrightarrow & H  \\
		\downarrow & & \downarrow \\
		A & \Longrightarrow & M,
	\end{tikzcd}
\end{eqnarray}
such that the (vertical) multiplication by scalars is by Lie algebroid maps, following the formulation in \cite{bursztyn2016vector}, which is equivalent to Mackenzie's original definition. The (vertical) dual of a (horizontal) VB-algebroid is naturally a (horizontal) VB-algebroid. A VB-algebroid is related by (horizontal) duality to a \textbf{Poisson double vector bundle}: a double vector bundle $\Omega$ equipped with a Poisson structure $\pi$ that is double linear, i.e. is linear with respect to both horizontal and vertical vector bundle structures. 

\vspace{1mm}

\begin{exemplo}
	The tangent bundle of any vector bundle $E \rightarrow M$ is a VB-algebroid, and if $(\pi, E) \rightarrow M$ is a Poisson vector bundle, its cotangent bundle is a VB-algebroid.
\end{exemplo}

\vspace{1mm}

\begin{exemplo}
    For a Lie algebroid $A \Rightarrow M$, $TA \Rightarrow TM$ and $T^*A \Rightarrow A^*$ are VB-algebroids over $A \Rightarrow M$. These are dual VB-algebroids in the sense just described, i.e. the Lie algebroids $T^*A \Rightarrow A^*$ and $(TA)^{*_{A}} \Rightarrow A^*$ are isomorphic (see \cite{bursztyn2016vector}).
\end{exemplo}

\vspace{1mm}

\begin{proposicao}(\cite[Corollary 3.4.4]{bursztyn2016vector})\label{drinfelddoubleevbalgebroid}
   {\em  Given a VB-algebroid $\Omega \Rightarrow H$ over $A \Rightarrow M$ and a Lie algebroid map $(\Phi, \phi) : (\tilde{A} \Rightarrow \tilde{M}) \to (A \Rightarrow M)$, then the pullback vector bundles $\Phi^*\Omega \Rightarrow \phi^*H$ form a VB-algebroid structure over $\tilde{A} \Rightarrow \tilde{M}$.}
\end{proposicao}

\begin{obs}\label{dualvbalgbd}
    Given $\Omega$ a VB-algebroid as in \eqref{defvbalgebroid}, a double vector subbundle $\Omega' \subseteq \Omega$ is a VB-subalgebroid if and only if its annihilator 
    $$
    \begin{tikzcd}[
		column sep={2em,between origins},
		row sep={1.3em,between origins},]
		ann(\Omega') \,\,\,\,\,\,\,\,\,\,\,\, & \Longrightarrow & \,\,\,\,\,\,\,\,\,\,\,\, ann(C')  \\
		\downarrow \,\,\,\,\,\,\,\,\,\,\,\, & & \,\,\,\,\,\,\,\,\,\,\,\, \downarrow \\
		A' \,\,\,\,\,\,\,\,\,\,\,\, & \Longrightarrow & \,\,\,\,\,\,\,\,\,\,\,\, M',
	\end{tikzcd}
    $$
    where $C'$ is the core of $\Omega'$, is a VB-subalgebroid of $\Omega^{*_A}$ (see \cite{bursztyn2016vector}).
\end{obs}

\vspace{1mm}

A \textbf{double Lie algebroid} \cite{mackenzie2006ehresmann} is a double vector bundle $\Omega$ which is both a horizontal and a vertical VB-algebroid,
	\begin{eqnarray}\label{doublealgebroid}
	\begin{tikzcd}[
		column sep={2em,between origins},
		row sep={1.3em,between origins},]
		\Omega  & \Longrightarrow & H  \\
		\Downarrow & & \Downarrow \\
		A & \Longrightarrow & M,
	\end{tikzcd}
\end{eqnarray}
whose (vertical) dual is a Poisson VB-algebroid. A double Lie algebroid is related by duality to a \textbf{Poisson VB-algebroid}: a VB-algebroid $\Omega$, as in \eqref{defvbalgebroid}, with a Poisson structure $\pi \in \mathfrak{X}(\Omega)$, such that $(\Omega, \pi) \Rightarrow H$ is a Poisson algebroid and $(\Omega, \pi) \rightarrow A$ is a Poisson vector bundle. We also refer to Poisson VB-algebroids as \textbf{PVB-algebroids}. 

\begin{examples}
	Given a Lie algebroid $A \Rightarrow M$, its tangent bundle is naturally a double Lie algebroid. Also, the VB-algebroid $T^*A \Rightarrow A^*$ over $A \Rightarrow M$ together with the canonical symplectic structure is a PVB-algebroid. If $(A, \pi) \Rightarrow M$ is a Poisson algebroid, then $TA \Rightarrow TM$ is a PVB-algebroid with respect to the tangent lift of $\pi$.
\end{examples}

\vspace{1mm}

\begin{definicao}
	Given a double vector bundle $\Omega$, we say that $(\Omega, \Omega^*)$ is a \textbf{double Lie bialgebroid} if $\Omega$ and $\Omega^{*_A}$ are double Lie algebroids,
	\begin{eqnarray} \label{doublebialgebroid}
	\begin{tikzcd}[
		column sep={2em,between origins},
		row sep={1.3em,between origins},]
		\Omega &\Longrightarrow & H \\
		\Downarrow & & \Downarrow \\
		A & \Longrightarrow & M, 
	\end{tikzcd}
	\,\,\,\,\,\,\,\,\,\,\,\,\,\,\,
	\begin{tikzcd}[
		column sep={2em,between origins},
		row sep={1.3em,between origins},]
		\,\, \Omega^{*_A} \,\, & \Longrightarrow & \,\, C^* \\
		\Downarrow \,\, & & \Downarrow \\
		A \,\, & \Longrightarrow & M,
	\end{tikzcd}
    \end{eqnarray}
	such that the horizontal VB-algebroids are in duality and the vertical Lie algebroids are compatible in the sense that $(\Omega \Rightarrow A, \Omega^{*_A} \Rightarrow A)$ is a Lie bialgebroid.
\end{definicao}

It is a direct consequence of the duality between double Lie algebroids and Poisson VB-algebroids that a double Lie bialgebroid $(\Omega, \Omega^*)$ is equivalent to a \textbf{Poisson double algebroid}: a double Lie algebroid $\Omega$ equipped with a Poisson structure $\pi \in \mathfrak{X}^2(\Omega)$ such that $\Omega \Rightarrow H$ and $\Omega \Rightarrow A$ are Poisson algebroids.

\begin{exemplo}
	A Poisson structure $\pi$ on a Lie algebroid $A \Rightarrow M$ defines a Poisson algebroid (or equivalently, a Lie bialgebroid $(A, A^*)$) if and only if $T^*A$ is a double Lie algebroid \cite[Thm 4.1]{mackenzie2006ehresmann}. This \textit{cotangent double of a Lie bialgebroid} is compatible with the canonical symplectic structure on $T^*A$, hence it is a Poisson double algebroid.
\end{exemplo}

\section{Manin triples for double Lie bialgebroids}\label{section4}

In this section, we provide a characterization of double Lie bialgebroids in terms of Manin triples. We start by recalling the notion of Courant algebroids over Lie algebroids from \cite{li2012courant}.

\subsection{LA-Courant algebroids}

In order to define Courant algebroids over Lie algebroids, we start by defining Courant algebroids over vector bundles: a \textbf{VB-Courant algebroid} is a double vector bundle
	$$
	\begin{tikzcd}[
		column sep={2em,between origins},
		row sep={1.4em,between origins},]
		\mathbb{E} & \longrightarrow & H  \\
		\downarrowtail & & \downarrow \\
		E & \longrightarrow & M
	\end{tikzcd}
	$$
such that  $\mathbb{E} \rightarrowtail E$ is a Courant algebroid and $gr(+_{\mathbb{E}/H}) \subseteq \mathbb{E} \times \overline{\mathbb{E} \times \mathbb{E}}$ is Dirac. A \textbf{VB-Dirac structure} on $\mathbb{E}$ is a double vector subbundle $L \rightarrow W$ over $E \rightarrow M$ such that $L \subseteq \mathbb{E}$ is a Dirac structure. 

\begin{exemplo}
	Given a Courant algebroid $\mathbb{E} \rightarrowtail M$, its tangent lift $T\mathbb{E} \rightarrowtail TM$ is naturally a Courant algebroid (see \cite{boumaiza2009relevement}), and fits into a VB-Courant algebroid over $\mathbb{E} \rightarrow M$ (see \cite[Proposition 3.4.1]{li2012courant}).
\end{exemplo}

\vspace{1mm}

Suppose that 
\begin{eqnarray}\label{lacourant}
	\begin{tikzcd}[
		column sep={2em,between origins},
		row sep={1.4em,between origins},]
		\mathbb{A} & \Longrightarrow & H  \\
		\downarrowtail & & \downarrow \\
		A & \Longrightarrow & M
	\end{tikzcd}
\end{eqnarray}
is both a VB-Courant algebroid and a VB-algebroid. Since it is a VB-algebroid, dualizing $\mathbb{A}$ over $H$ yields a double linear Poisson structure $\pi \in \mathfrak{X}(\mathbb{A}^{*_H})$, and therefore the map $\pi^{\sharp} : T^*(\mathbb{A}^{*_H}) \longrightarrow T\mathbb{A}^{*_H}$ defines the following morphism of triple vector bundles
$$
\begin{minipage}{.23\textwidth}
	\begin{tikzcd}[back line/.style={densely dotted}, row sep=0.6em, column sep=0.6em]
		& T^*\mathbb{A}^{*_H} \ar{dl}[swap]{} \ar{rr} \ar[back line]{dd}
		& & \mathbb{A} \ar{dd}{} \ar{dl}[swap,sloped,near start]{} \\
		\mathbb{A}^{*_H} \ar[crossing over]{rr}[near start]{} \ar{dd}[swap]{} 
		& & H \\
		& \mathbb{A}^{*_H*_{C^*}} \ar[back line]{rr} \ar[back line]{dl} 
		& & A \ar{dl} \\
		C^* \ar{rr} & & M \ar[crossing over, leftarrow]{uu}
	\end{tikzcd}
\end{minipage}  \qquad \qquad 
\begin{minipage}{.13\textwidth}
	\begin{tikzcd}
		\draw[->] (0,0,0) -- (1.6,0, 0) node[midway,above] {\pi^{\sharp}};
	\end{tikzcd}
\end{minipage}
\begin{minipage}{.34\textwidth}
	\begin{tikzcd}[back line/.style={densely dotted}, row sep=0.6em, column sep=0.6em]
		& T\mathbb{A}^{*_H} \ar{dl}[swap]{} \ar{rr} \ar[back line]{dd}
		& & TH \ar{dd}{} \ar{dl}[swap,sloped,near start]{} \\
		\mathbb{A}^{*_H} \ar[crossing over]{rr}[near start]{} \ar{dd}[swap]{} 
		& & H \\
		& TC^* \ar[back line]{rr} \ar[back line]{dl} 
		& & TM. \ar{dl} \\
		C^* \ar{rr} & & M \ar[crossing over, leftarrow]{uu}
	\end{tikzcd}
\end{minipage}
$$
By dualizing the above setup along the horizontal direction, one obtains a relation of triple vector bundles $\Pi_{\mathbb{A}} : (T^*\mathbb{A}^{*_H})^{*_\mathbb{A}} \dashrightarrow T\mathbb{A}$, where $\Pi_{\mathbb{A}} = ann^{\natural}(gr(\pi^{\sharp}))$. By using the reversal isomorphism (\cite[Theorem $5.5$]{mackenzie1994lie}), we can identify $(T^*\mathbb{A}^{*_H})^{*_\mathbb{A}} \simeq (T^*\mathbb{A})^{*_\mathbb{A}} = T\mathbb{A}$, and then $\Pi_{\mathbb{A}}$ is a relation between the following triple vector bundles:
$$
\begin{tikzcd}[back line/.style={densely dotted}, row sep=0.6em, column sep=0.6em]
	& T\mathbb{A} \ar{dl}[swap]{} \ar{rr} \ar[back line]{dd}
	& & \mathbb{A} \ar{dd}{} \ar{dl}[swap,sloped,near start]{} \\
	TH \ar[crossing over]{rr}[near start]{} \ar{dd}[swap]{} 
	& & H \\
	& TA \ar[back line]{rr} \ar[back line]{dl} 
	& & A \ar{dl} \\
	TM \ar{rr} & & M \ar[crossing over, leftarrow]{uu}
\end{tikzcd} \,\,\,
\begin{tikzcd}
	\draw[dashed, ->] (0,0,0) -- (1.8,0, 0) node[midway,above] {\Pi_{\mathbb{A}}};
\end{tikzcd}\,\,\,
\begin{tikzcd}[back line/.style={densely dotted}, row sep=0.6em, column sep=0.6em]
	& T\mathbb{A} \ar{dl}[swap]{} \ar{rr} \ar[back line]{dd}
	& & TH \ar{dd}{} \ar{dl}[swap,sloped,near start]{} \\
	\mathbb{A} \ar[crossing over]{rr}[near start]{} \ar{dd}[swap]{} 
	& & H \\
	& TA \ar[back line]{rr} \ar[back line]{dl} 
	& & TM. \ar{dl} \\
	A \ar{rr} & & M \ar[crossing over, leftarrow]{uu}
\end{tikzcd}
$$

\begin{definicao}
	An \textbf{LA-Courant algebroid} is a double vector bundle with both a VB-Courant algebroid structure and a VB-algebroid structure, as in \eqref{lacourant}, such that $\Pi_{\mathbb{A}} : T\mathbb{A} \dashrightarrow T\mathbb{A}$ is a Courant relation, where $T\mathbb{A} \rightarrowtail TA$ is the tangent prolongation of $\mathbb{A} \rightarrowtail A$ (see \cite{boumaiza2009relevement}). A VB-Dirac structure $L \subseteq \mathbb{A}$ is called an \textbf{LA-Dirac structure} if it is also a Lie subalgebroid of $\mathbb{A} \Rightarrow H$. If $L_1, L_2 \subseteq \mathbb{A}$ are transverse and LA-Dirac structures, we say that $(\mathbb{A}, L_1, L_2)$ is an \textbf{LA-Manin triple}.
\end{definicao}

\begin{exemplo}
	Given an LA-Courant algebroid $\mathbb{E} \rightarrowtail M$, Li-Bland \cite{li2012courant} showed that its Courant tangent prolongation $T\mathbb{E} \rightarrowtail TM$ is an LA-Courant algebroid. Also, if $L \subseteq \mathbb{E}$ is a Dirac structure, then $TL$ is an LA-Dirac structure.
\end{exemplo}

\vspace{1mm}

The next Proposition says that the compatibility between the bilinear form and the Lie algebroid structure on an LA-Courant algebroid can be expressed in terms of a Lie algebroid morphism.

\begin{proposicao}\label{algbmorf}(\cite[Proposition 5.1.1]{li2012courant})
   {\em  Suppose that $\mathbb{A}$ is an LA-Courant algebroid as in \eqref{lacourant}. Then the map $(\mathbb{A} \Rightarrow H) \rightarrow (\mathbb{A}^{*_A} \Rightarrow C^*)$ induced by the fiber metric is a Lie algebroid morphism.}
\end{proposicao}

\vspace{1mm}

\subsection{The relation in the definition of LA-Courant algebroids}\label{subsectiontherelation}

As seen in the previous section, the definition of an LA-Courant algebroid is intricate due to its formulation in terms of triple vector bundles; nevertheless, the relation involved in that definition can be considered even if we have only a Lie algebroid. We start this section by discussing this relation for Lie algebroids, and then we will introduce additional structures and observe what these new structures and their compatibilities yield in terms of this relation.

\vspace{3mm}

Let $A \Rightarrow M$ be a Lie algebroid with anchor $\rho$ and  $\pi^{\sharp}_{A \Rightarrow M} : T^*A^* \rightarrow TA^*$ its linear Poisson structure on $A^* \rightarrow M$. We associate with this Lie algebroid the following relation
\begin{eqnarray}\label{relation1}
	\Pi_A := \{ (V, W) \in TA \times TA \,\, \vert\ \, \langle V, \phi \rangle_A = \langle W, \pi^{\sharp}_{A \Rightarrow M} (\phi) \rangle_{TM} \,\, \forall \, \phi \in T^*A^* \}, \nonumber
\end{eqnarray}
where we consider the identifications $T^*A^* \simeq T^*A$ (given by the reversal isomorphism in \cite[Theorem $5.5$]{mackenzie1994lie}) and $TA^{*_M} \simeq (TA)^{*_{TM}}$, with $\langle \,\, , \, \rangle_{A}$ and $\langle \,\, , \, \rangle_{TM}$ denoting the pairings between $TA$ and $T^*A$, and between $TA$ and $(TA)^{*_{TM}}$, respectively. Note that, under the reversal isomorphism, $\Pi_A$ is identified with $ann^{\natural}(gr(\pi^{\sharp}_{A \Rightarrow M}))$. Indeed, the map $\pi^\sharp_{A \Rightarrow M}$ fits into the following diagram
$$
	\begin{tikzcd}[back line/.style={densely dotted}, row sep=0.8em, column sep=2.3em]
		& A^* \ar[back line]{rr}{id_{A^*}} \ar{dd}
		& & A^* \ar{dd}{}  \\
		T^*A^* \ar{ur}[swap]{} \ar[crossing over, back line, near end]{rr}{\pi^{\sharp}_{A \Rightarrow M}} \ar{dd}[swap]{} 
		& & TA^* \ar{ur}[swap,sloped,near start]{} \\
		& M \ar[back line]{rr}[near start]{id_M} 
		& & M, \\
		A \ar[back line]{rr}{\rho} \ar{ur} & & TM \ar{ur} \ar[crossing over, leftarrow]{uu}
	\end{tikzcd}
$$
where the map between the cores is $-\rho^* : T^*M \rightarrow A^*$, with $\rho^* : T^*M \rightarrow A^*$ the dual map of the anchor $\rho : A \rightarrow TM$. Then, dualizing $gr(\pi^{\sharp}_{A \Rightarrow M})$ on the vertical direction we get the double vector bundle
$$
	\begin{tikzcd}[
		column sep={3em,between origins},
		row sep={1.6em,between origins},]
		ann^{\natural}(gr(\pi^{\sharp}_{A \Rightarrow M})) \qquad \qquad & \longrightarrow & gr(-\rho) \\
		\downarrow \qquad \qquad & & \downarrow \\
		gr(\rho) \qquad \qquad & \longrightarrow & gr(id_M),
	\end{tikzcd}
$$
where we use that $ann^{\natural}(gr(-\rho^*)) = gr(-\rho)$. By considering the identifications $(T^*A^*)^{*_A}$ $\simeq (T^*A)^{*_A} \simeq TA$ and $(TA^{*_M})^{*_{TM}} \simeq TA$, and since the reversal isomorphism $T^*A^* \simeq T^*A$ in \cite[Theorem $5.5$]{mackenzie1994lie} induces $-id$ on the core, we get that the last 
double vector bundle displayed above becomes
$$
	\begin{tikzcd}[
		column sep={2.8em,between origins},
		row sep={1.6em,between origins},]
		\Pi_A & \longrightarrow & \,\, gr(\rho)  \\
		\downarrow & & \,\, \downarrow \\
		gr(\rho) & \longrightarrow & \,\, gr(id_M).
	\end{tikzcd}
$$
Note that the core of $\Pi_A$ is $gr(id_A) \simeq A$. Since in this case we can consider the annihilator in more than one direction, we will denote the above annihilator by 
$$
	ann^{\natural}_{A \times TM}(gr(\pi^{\sharp}_{A \Rightarrow M})).
$$
Evidencing the reversal isomorphism, the relation $\Pi_A$ can be written as
\begin{eqnarray}\label{relation2}
	\Pi_A := \{ (V, W) \in TA \times TA \,\, \vert\ \, \langle V, \psi \rangle_A = \langle W, \pi^{\sharp}_{A \Rightarrow M} (R^{-1}(\psi)) \rangle_{TM} \,\, \forall \, \psi \in T^*A \}, \nonumber
\end{eqnarray}
i.e., $\Pi_A = ann^{\natural}(gr(\pi^{\sharp} \circ R^{-1}))$, where $R^{-1} : T^*A \rightarrow T^*A^*$ is the inverse of the reversal isomorphism.

\vspace{1mm}

Now, given a VB-algebroid
$$
	\begin{tikzcd}[
		column sep={2em,between origins},
		row sep={1.3em,between origins},]
		\Omega & \Longrightarrow & H  \\
		\downarrow & & \downarrow \\
		A & \Longrightarrow & M,
	\end{tikzcd}
$$
we will understand in detail the triple vector bundle morphism given by the linear Poisson structure associated to the Lie algebroid $\Omega \Rightarrow H$, and the associated triple vector bundle $\Pi_\Omega$. Dualizing the double vector bundle $\Omega$ over $H$, we get the following triple vector bundle morphism:
$$
\begin{minipage}{.23\textwidth}
	\begin{tikzcd}[back line/.style={densely dotted}, row sep=0.6em, column sep=0.6em]
		& T^*\Omega^{*_H} \ar{dl}[swap]{} \ar{rr} \ar[back line]{dd}
		& & \Omega \ar{dd}{} \ar{dl}[swap,sloped,near start]{} \\
		\Omega^{*_H} \ar[crossing over]{rr}[near start]{} \ar{dd}[swap]{} 
		& & H \\
		& \Omega^{* _H*_{C^*}} \ar[back line]{rr} \ar[back line]{dl} 
		& & A \ar{dl} \\
		C^* \ar{rr} & & M \ar[crossing over, leftarrow]{uu}
	\end{tikzcd}
\end{minipage}  \qquad \qquad 
\begin{minipage}{.13\textwidth}
	\begin{tikzcd}
		\draw[->] (0,0,0) -- (1.6,0, 0) node[midway,above] {\pi_{\Omega}^{\sharp}};
	\end{tikzcd}
\end{minipage}
\begin{minipage}{.32\textwidth}
	\begin{tikzcd}[back line/.style={densely dotted}, row sep=0.6em, column sep=0.6em]
		& T\Omega^{*_H} \ar{dl}[swap]{} \ar{rr} \ar[back line]{dd}
		& & TH \ar{dd}{} \ar{dl}[swap,sloped,near start]{} \\
		\Omega^{*_H} \ar[crossing over]{rr}[near start]{} \ar{dd}[swap]{} 
		& & H \\
		& TC^* \ar[back line]{rr} \ar[back line]{dl} 
		& & TM. \ar{dl} \\
		C^* \ar{rr} & & M \ar[crossing over, leftarrow]{uu}
	\end{tikzcd}
\end{minipage}
$$
By writing the expression of the Poisson structure in coordinates and making the respective projections, we find that $\pi_{\Omega}^{\sharp}$ induces the following maps on the sides and core bundles:
\begin{multicols}{2}
  \begin{itemize}[left=10pt, itemsep=0.5em]
    \item $f_{D_1} = \rho_{\Omega} : \Omega \rightarrow TH$,
    \item $f_{D_2} = id_{\Omega^{*_H}} : \Omega^{*_H} \rightarrow \Omega^{*_H}$,
    \item $f_{D_3} = \rho_{\Omega^{*_H*_{C^*}}} : \Omega^{*_H*_{C^*}} \rightarrow TC^*$,
    \item $f_{V_1} = id_{C^*} : C^* \rightarrow C^*$,
    \item $f_{V_2} = \rho_{A} : A \rightarrow TM$,
    \item $f_{V_3} = id_H : H \rightarrow H$,
    \item $f_{M} = id_M : M \rightarrow M$,
    \item $f_{C_1} = - \rho^*_{\Omega^{*_H*_{C^*}}} : T^*C^* \rightarrow \Omega^{*_H}$,
    \item $f_{C_2} = \pi_A^{\sharp} : T^*A^* \rightarrow TA^*$,
    \item $f_{C_3} = -\rho^*_{\Omega} : T^*H \rightarrow \Omega^{*_H}$,
    \item $f_{K_1} = \rho_{\Omega}\vert_C : C \rightarrow H$,
    \item $f_{K_2} = id_{A^*} : A^* \rightarrow A^*$,
    \item $f_{K_3} = - (\rho_{\Omega}\vert_C)^* : H^* \rightarrow C^*$,
    \item $f_{C} = - \rho_{A}^* : T^*M \rightarrow A^*$,
  \end{itemize}
\end{multicols}
\noindent where the indices of the functions indicate which side or core we are referring to, following the convention for an arbitrary triple vector bundle as in Appendix \ref{apendicetriplevb}. 

\vspace{1mm}

Composing $\pi_{\Omega}^{\sharp}$ with the triple vector bundle morphism given by the inverse of the reversal isomorphism $R^{-1}$ (see Appendix \ref{apendicetriplevb}) and taking the dual over $\Omega \times TH$, we get that $\Pi_{\Omega}$ is the following triple vector bundle
$$
\begin{minipage}{.42\textwidth}
	\begin{tikzcd}[back line/.style={densely dotted}, row sep=0.6em, column sep=0.6em]
		& \Pi_{\Omega} \ar{dl}[swap]{} \ar{rr} \ar[back line]{dd}
		& & gr(\rho_{\Omega}) \ar{dd}{} \ar{dl}[swap,sloped,near start]{} \\
		gr(\rho_{\Omega}) \ar[crossing over]{rr}[near start]{} \ar{dd}[swap]{} 
		& & gr(id_H) \\
		& \Pi_A \ar[back line]{rr} \ar[back line]{dl} 
		& & gr(\rho_{A}), \ar{dl} \\
		gr(\rho_{A}) \ar{rr} & & gr(id_M) \ar[crossing over, leftarrow]{uu}
	\end{tikzcd}
\end{minipage}
$$
with ultracore $C = gr(id_C)$ and the following cores:
\begin{multicols}{2}
  \begin{itemize}[left=10pt, itemsep=0.5em]
    \item $C_1 = ann^{\natural}_{TM \times A}(gr(\rho_{\Omega^{* _A}}))$,
    \item $C_2 = ann^{\natural}_{A \times TM}(gr(\rho_{\Omega^{* _A}}))$,
    \item $C_3 = gr(id_\Omega)$,
    \item $K_1 = gr(\rho_{\Omega}\vert_C)$,
    \item $K_2 = gr(\rho_{\Omega}\vert_C)$,
    \item $K_3 = gr(id_A)$.
  \end{itemize}
\end{multicols} 

\vspace{1mm}

Consider the following VB-algebroids:
	$$
	\begin{tikzcd}[
		column sep={2em,between origins},
		row sep={1.3em,between origins},]
		\Omega & \Longrightarrow & H  \\
		\downarrow & & \downarrow \\
		A & \Longrightarrow & M,
	\end{tikzcd}
\qquad
	\begin{tikzcd}[
		column sep={2em,between origins},
		row sep={1.3em,between origins},]
		\Omega' & \Longrightarrow & H'  \\
		\downarrow & & \downarrow \\
		A & \Longrightarrow & M.
	\end{tikzcd}
	$$
We want to relate $\Pi_{\Omega \oplus_A \Omega'}$ and $\Pi_{\Omega} \oplus_{\Pi_A} \Pi_{\Omega'}$: 

\begin{proposicao}\label{propsumrelations}
	{\em Given VB-algebroids $\Omega$ and $\Omega'$ as above, the triple vector bundles $\Pi_{\Omega \oplus_A \Omega'}$ and $\Pi_\Omega \oplus_{\Pi_A} \Pi_{\Omega'}$ are canonically isomorphic.}
\end{proposicao}
\begin{proof}
Consider the following isomorphisms of triple vector bundles
\begin{eqnarray}
	F : T(\Omega \oplus_A \Omega') \times T(\Omega \oplus_A \Omega') &\longrightarrow& (T\Omega \times T\Omega) \oplus_{TA \times TA} (T\Omega' \times T\Omega') \nonumber \\
	(V, W) &\longmapsto& (dp_{\Omega}(V), dp_{\Omega}(W)) \oplus (dp_{\Omega'}(V), dp_{\Omega'}(W)), \nonumber
\end{eqnarray}
and
\begin{eqnarray}
	f : T(\Omega \oplus_A \Omega') &\longrightarrow& T\Omega \oplus_{TA} T\Omega' \nonumber \\
	V &\longmapsto& dp_{\Omega}(V) \oplus dp_{\Omega'}(V), \nonumber
\end{eqnarray}
where $p_\Omega : \Omega \oplus_A \Omega' \rightarrow \Omega$ and $p_{\Omega'} : \Omega \oplus_A \Omega' \rightarrow \Omega'$ are the usual projections. Since the Lie algebroid structure on $\Omega \oplus_A \Omega'$ is simply a restriction of the product Lie algebroid, the following diagram commutes
\[
\begin{tikzcd}[column sep=7.5em, row sep=large]
(T\Omega \oplus_{TA} T\Omega')^{*_{\Omega \oplus_A \Omega'}} \arrow[r,"\pi^{\sharp}_{\Omega} \circ R_{\Omega}^{-1} \times \pi^{\sharp}_{\Omega'} \circ R_{\Omega'}^{-1}"] \arrow[d,"f^{*_{\Omega \oplus \Omega'}}"'] & (T\Omega \oplus_{TA} T\Omega')^{*_{TH \oplus TH'}} \arrow[d,"f^{*_{T(H \oplus H')}}"] \\
T^*(\Omega \oplus_A \Omega') \arrow[r,"\pi^{\sharp}_{\Omega \oplus \Omega'} \circ R_{\Omega \oplus \Omega'}^{-1}"'] & T(\Omega \oplus_A \Omega')^{*_{H \oplus H'}}.
\end{tikzcd}
\]
Therefore, it is straightfoward to verify that the restriction of $F$ to $\Pi_{\Omega \oplus_A \Omega'}$ has image $\Pi_\Omega \oplus_{\Pi_A} \Pi_{\Omega'}$, providing the desired isomorphism.
\end{proof}

\vspace{2mm}

Now we want to relate  $\Pi_{\Omega}$ and $\Pi_{\Omega^{*_A}}$ when we have a VB-algebroid and its dual VB-algebroid as follows:
$$
	\begin{tikzcd}[
		column sep={2em,between origins},
		row sep={1.3em,between origins},]
		\Omega & \Longrightarrow & H  \\
		\downarrow & {\scriptstyle C} & \downarrow \\
		A & \Longrightarrow & M,
	\end{tikzcd}
\qquad
	\begin{tikzcd}[
		column sep={2em,between origins},
		row sep={1.3em,between origins},]
		\Omega^{*_A} & \Longrightarrow & C^*  \\
		\downarrow & {\scriptstyle H^*} & \downarrow \\
		A & \Longrightarrow & M.
	\end{tikzcd}
$$
We provide some insight into their expected relationship. Note that the triple vector bundle associated to $\Pi_{\Omega^{*_A}}$ is given by 
    $$
\begin{minipage}{.47\textwidth}
	\begin{tikzcd}[back line/.style={densely dotted}, row sep=0.6em, column sep=0.6em]
		& \Pi_{\Omega^{*_A}} \ar{dl}[swap]{} \ar{rr} \ar[back line]{dd}
		& & gr(\rho_{\Omega^{*_A}}) \ar{dd}{} \ar{dl}[swap,sloped,near start]{} \\
		gr(\rho_{\Omega^{*_A}}) \ar[crossing over]{rr}[near start]{} \ar{dd}[swap]{} 
		& & gr(id_{C^*}) \\
		& \Pi_A \ar[back line]{rr} \ar[back line]{dl} 
		& & gr(\rho_{A}). \ar{dl} \\
		gr(\rho_{A}) \ar{rr} & & gr(id_M) \ar[crossing over, leftarrow]{uu}
	\end{tikzcd}
\end{minipage}
$$
By using the duality theory of triple vector bundles (see Appendix \ref{apendicetriplevb}), and observing that $\Pi_{\Omega}$ has cores $C_1 = ann^{\natural}_{TM \times A}(gr(\rho_{\Omega^{* _A}}))$, $C_2 = ann^{\natural}_{A \times TM}(gr(\rho_{\Omega^{* _A}}))$ and ultracore $gr(id_C)$, we see that the triple vector bundle $ann^{\natural}_{TA \times TA}(\Pi_{\Omega})$ fits into the following diagram:
$$
\begin{minipage}{.69\textwidth}
	\begin{tikzcd}[back line/.style={densely dotted}, row sep=0.6em, column sep=0.6em]
		& ann^{\natural}_{TA \times TA} (\Pi_{\Omega}) \ar{dl}[swap]{} \ar{rr} \ar[back line]{dd}
		& & gr(\rho_{\Omega^{*_A}}) \ar{dd}{} \ar{dl}[swap,sloped,near start]{} \\
		gr(\rho_{\Omega^{*_A}}) \ar[crossing over]{rr}[near start]{} \ar{dd}[swap]{} 
		& & ann^{\natural}(gr(id_C)) \\
		& \Pi_A \ar[back line]{rr} \ar[back line]{dl} 
		& & gr(\rho_{A}). \ar{dl} \\
		gr(\rho_{A}) \ar{rr} & & gr(id_M) \ar[crossing over, leftarrow]{uu}
	\end{tikzcd}
\end{minipage}
$$
This discussion motivates the following precise relationship between $\Pi_{\Omega}$ and $\Pi_{\Omega^{*_A}}$.

\begin{proposicao}\label{proprelationdual}
   { \em Let $\Omega$ be a VB-algebroid and $\Omega^{*_A}$ its dual VB-algebroid as above. Under the natural isomorphism $T\Omega^{*_A} \;\simeq\; (T\Omega)^{*_{TA}}$, the triple vector subbundles $\Pi_{\Omega^{*_A}} \subseteq T\Omega^{*_A} \times T\Omega^{*_A}$ and $ann^{\natural}_{TA \times TA} (\Pi_{\Omega}) \subseteq (T\Omega)^{*_{TA}} \times (T\Omega)^{*_{TA}}$ are equal.
   }
\end{proposicao}
\begin{proof}
	Let $\pi_\Omega$ and $\pi_{\Omega^{*_A}}$ denote the Poisson structures on the double vector bundles $\Omega^{*_H}$ and $\Omega^{*_A*_{C^*}}$, respectively. These structures are related by the following equation
    \[
    \pi_{\Omega^{*_A}}^{\sharp} = dZ_H \circ \pi_\Omega^{\sharp} \circ (dZ_H)^{*_{\Omega^{*_H}}},
    \]
    where $dZ_H$ is the tangent lift of the isomorphism $Z_H : \Omega^{*_H} \to \Omega^{*_A*_{C^*}}$ (see Appendix \ref{apendicetriplevb}). Also, we denote by $R_{\Omega \to H}$ and $R_{\Omega^{*_A} \to C^*}$ the reversal isomorphisms arising from the indicated vector bundles. Given $(\mu, \nu) \in \Pi_{\Omega^{*_A}}$, for all $\beta \in T^*\Omega^{*_A}$ and $\alpha := R_{\Omega \to H} \circ (dZ_H)^{*_{\Omega^{*_H}}} \circ R_{\Omega^{*_A} \to C^*}^{-1} (\beta) \in T^*\Omega$, we have
    \begin{eqnarray}
        \langle \mu, \beta \rangle_{\Omega^{*_A}} &=& \langle \nu, \pi_{\Omega^{*_A}}^{\sharp} \circ R_{\Omega^{*_A} \to C^*}^{-1} (\beta) \rangle_{TC^*} \nonumber \\
        &=& \langle \nu, \pi_{\Omega^{*_A}}^{\sharp} \circ ((dZ_H)^{*_{\Omega^{*_H}}})^{-1} \circ R_{\Omega \to H}^{-1} (\alpha) \rangle_{TC^*} \nonumber \\
        &=& \langle \nu, dZ_H \circ \pi_{\Omega}^{\sharp} \circ R_{\Omega \to H}^{-1} (\alpha) \rangle_{TC^*} \nonumber \\
        &=& \vertbold \nu, \pi_{\Omega}^{\sharp} \circ R_{\Omega \to H}^{-1} (\alpha) \vertbold \nonumber \\
        &=& \langle \pi_{\Omega}^{\sharp} \circ R_{\Omega \to H}^{-1} (\alpha), \tilde{V} \rangle_{TH} - \langle \nu, \tilde{V} \rangle_{TA}, \nonumber
    \end{eqnarray}
    for all $\tilde{V} \in T\Omega$ compatible with the respective projections, where in the last two equalities we use the definition of the isomorphism $Z_H$ (see Appendix \ref{apendicetriplevb}). Then, given $(U,V) \in \Pi_\Omega$, since $\langle U , \alpha \rangle_\Omega = \langle V, \pi_{\Omega}^{\sharp} \circ R_{\Omega \to H}^{-1} (\alpha) \rangle_{TH}$ for all $\alpha \in T^*\Omega$, we have
    \begin{eqnarray}
        \langle \nu, V \rangle_{TA} &=& \langle U , \alpha \rangle_\Omega - \langle \mu, \beta \rangle_{\Omega^{*_A}} \nonumber \\
        &=& \langle U , \alpha \rangle_\Omega - \langle \mu,  R_{\Omega^{*_A} \to C^*} \circ ((dZ_H)^{*_{\Omega^{*_H}}})^{-1} \circ R_{\Omega \to H}^{-1}(\alpha) \rangle_{\Omega^{*_A}} \nonumber \\
        &=& \langle U , \alpha \rangle_\Omega - \langle \mu,  -_{\Omega^{*_A}} R_{\Omega \to H}^{-1}(\alpha) \rangle_{\Omega^{*_A}} \nonumber \\
        &=& \langle U , \alpha \rangle_\Omega + \langle \mu, R_{\Omega \to H}^{-1}(\alpha) \rangle_{\Omega^{*_A}} \nonumber \\
        &=& \langle \mu, U \rangle_{TA}, \nonumber
    \end{eqnarray}
    where we use Proposition \ref{propreversalcommute} and \cite[Theorem 9.5.1]{mackenzie2005general} (see Appendix \ref{apendicetriplevb}). Therefore, $(\mu, \nu) \in ann^{\natural}_{TA \times TA}(\Pi_\Omega)$. The converse follows simply by reversing the preceding arguments. 
\end{proof}

\vspace{1mm}

Let $\Omega$ be a VB-algebroid and $\Omega^{*_A}$ its dual VB-algebroid. Consider the following isomorphisms of VB-algebroids:
\begin{eqnarray}\label{eq:phi}
\varphi: T\Omega^{*_A} \times T\Omega^{*_A} &\longrightarrow& T\Omega^{*_A} \times T\Omega^{*_A} \\
(\xi, \eta) &\longmapsto& (\xi, -\eta), \nonumber
\end{eqnarray}
and 
\begin{eqnarray}\label{eq:Phi}
\Phi : T(\Omega {\oplus} \Omega^{*_A}) {\times} T(\Omega {\oplus} \Omega^{*_A}) & {\longrightarrow} & (T\Omega {\times} T\Omega) {\oplus} (T\Omega^{*_A} {\times} T\Omega^{*_A}) \\
(V, W) & {\longmapsto} & (dp_\Omega(V), dp_\Omega(W)) {\oplus} (dp_{\Omega^{*_A}}(V), - dp_{\Omega^{*_A}}(W)), \nonumber
\end{eqnarray}
where the minus sign is from the vector bundle $T\Omega^{*_A} \rightarrow TA$, and $p_\Omega : \Omega \oplus_A \Omega^{*_A} \rightarrow \Omega$, $p_{\Omega^{*_A}} : \Omega \oplus_A \Omega^{*_A} \rightarrow \Omega^{*_A}$ are the usual projections. We will need the following result, that follows immediately from the definitions and from Proposition \ref{propsumrelations} and Proposition \ref{proprelationdual}.

\begin{proposicao}\label{prop:12}
    The following holds:
    \begin{itemize}[left=13pt, itemsep=0.3em]
\item[(a)] $\varphi(\Pi_{\Omega^{*_A}}) = ann_{TA \times TA}(\Pi_{\Omega});$
\item[(b)] $\Phi(\Pi_{\Omega \oplus_A \Omega^{*_A}}) = \Pi_\Omega \oplus_{\Pi_A} ann_{TA \times TA}(\Pi_\Omega).$
    \end{itemize}
\end{proposicao}

\vspace{2mm}

Now, we aim to understand what happens to $\Pi_{\Omega}$ when $\Omega$ is a double Lie algebroid:
$$
\begin{tikzcd}[
		column sep={2em,between origins},
		row sep={1.3em,between origins},]
		\Omega & \Longrightarrow & H  \\
		\Downarrow & & \Downarrow \\
		A & \Longrightarrow & M.
\end{tikzcd}
$$
In this case, we can consider $\Pi_\Omega$ in two directions, so we will denote $\Pi_{\Omega \Rightarrow H}$ to indicate that we are considering the relation with respect to the Lie algebroid $\Omega \Rightarrow H$. 

\begin{proposicao}\label{proprelationsubalgbd}
	{\em Given a VB-algebroid $\Omega$ with another VB-algebroid structure as above, $\Omega$ is a double Lie algebroid if and only if $\Pi_{\Omega \Rightarrow H} \Rightarrow \Pi_A$ is a Lie subalgebroid of $T\Omega \times T\Omega \Rightarrow TA \times TA$.}
\end{proposicao}
\begin{proof}
Note that, to get $\Pi_{\Omega \Rightarrow H}$, we consider $\Omega^{*_H}$ with the Poisson structure $\pi_{\Omega \Rightarrow H}$, and compose the following triple vector bundle morphism with the inverse of the reversal isomorphism
$$
\begin{minipage}{.23\textwidth}
	\begin{tikzcd}[back line/.style={densely dotted}, row sep=0.6em, column sep=0.6em]
		& T^*\Omega^{*_H} \ar{dl}[swap]{} \ar{rr} \ar[Rightarrow]{dd}
		& & \Omega \ar[Rightarrow]{dd}{} \ar{dl}[swap,sloped,near start]{} \\
		\Omega^{*_H} \ar[crossing over]{rr}[near start]{} \ar[Rightarrow]{dd}[swap]{} 
		& & H \\
		& \Omega^{* _H*_{C^*}} \ar[back line]{rr} \ar[back line]{dl} 
		& & A \ar{dl} \\
		C^* \ar{rr} & & M \ar[crossing over, Leftarrow]{uu}
	\end{tikzcd}
\end{minipage}  \qquad \qquad 
\begin{minipage}{.13\textwidth}
	\begin{tikzcd}
		\draw[->] (0,0,0) -- (1.6,0, 0) node[midway,above] {\pi_{\Omega \Rightarrow H}^{\sharp}};
	\end{tikzcd}
\end{minipage}
\begin{minipage}{.32\textwidth}
	\begin{tikzcd}[back line/.style={densely dotted}, row sep=0.6em, column sep=0.6em]
		& T\Omega^{*_H} \ar{dl}[swap]{} \ar{rr} \ar[Rightarrow]{dd}
		& & TH \ar[Rightarrow]{dd}{} \ar{dl}[swap,sloped,near start]{} \\
		\Omega^{*_H} \ar[crossing over]{rr}[near start]{} \ar[Rightarrow]{dd}[swap]{} 
		& & H \\
		& TC^* \ar[back line]{rr} \ar[back line]{dl} 
		& & TM, \ar{dl} \\
		C^* \ar{rr} & & M \ar[crossing over, Leftarrow]{uu}
	\end{tikzcd}
\end{minipage}
$$
and then we dualize on the horizontal direction, getting the following triple vector bundle
$$
\begin{minipage}{.42\textwidth}
	\begin{tikzcd}[back line/.style={densely dotted}, row sep=0.6em, column sep=0.6em]
		& \Pi_{\Omega} \ar{dl}[swap]{} \ar{rr} \ar[Rightarrow]{dd}
		& & gr(\rho_{\Omega}) \ar[Rightarrow]{dd}{} \ar{dl}[swap,sloped,near start]{} \\
		gr(\rho_{\Omega}) \ar[crossing over]{rr}[near start]{} \ar[Rightarrow]{dd}[swap]{} 
		& & gr(id_H) \\
		& \Pi_A \ar[back line]{rr} \ar[back line]{dl} 
		& & gr(\rho_{A}). \ar{dl} \\
		gr(\rho_{A}) \ar{rr} & & gr(id_M) \ar[crossing over, Leftarrow]{uu}
	\end{tikzcd}
\end{minipage}
$$
Since $\Omega$ is a double Lie algebroid, its dual $\Omega^{*_H}$ is a Poisson VB-algebroid with the Poisson structure $\pi_{\Omega \Rightarrow H}$, then $gr(\pi^{\sharp}_{\Omega \Rightarrow H}) \Rightarrow gr(\rho_{\Omega^{*_H*_{C^*}}})$ is a Lie subalgebroid of $T^*\Omega^{*_H} \times T\Omega^{*_H} \Rightarrow \Omega^{*_H*_{C^*}} \times TC^*$. Therefore, since the inverse of the reversal isomorphism is by definition a VB-algebroid morphism over $Z_A : \Omega^{A^*} \rightarrow \Omega^{* _H*_{C^*}}$ (see Appendix \ref{apendicetriplevb}) and by Remark \ref{dualvbalgbd}, 
$$
ann_{\Omega \times TH}(gr(\pi_{\Omega}^{\sharp} \circ R^{-1})) \Rightarrow ann_{A \times TM} (gr(\pi_{A}^{\sharp} \circ R^{-1}))
$$ 
is a Lie subalgebroid of $T\Omega \times T\Omega \Rightarrow TA \times TA$. Since 
$$
\begin{tikzcd}[
		column sep={2em,between origins},
		row sep={1.4em,between origins},]
		T\Omega \,\, & \longrightarrow & \,\, TH  \\
		\Downarrow \,\, & & \,\, \Downarrow \\
		TA \,\, & \longrightarrow & \,\, TM
\end{tikzcd}
$$
is a VB-algebroid, the multiplication by scalars over $TH$ is a Lie algebroid morphism, and then $-_{TH}(id_{T\Omega})$, over $-_{TM}(id_{TA})$, is a Lie algebroid morphism between $T\Omega \Rightarrow TA$ and itself. So, the map $id_{T\Omega} \times (-_{TH}(id_{T\Omega}))$ is a Lie algebroid morphism between $T\Omega \times T\Omega \Rightarrow TA \times TA$ and itself, and takes $ann_{\Omega \times TH}(gr(\pi_{\Omega}^{\sharp} \circ R^{-1}))$ to $\Pi_{\Omega \Rightarrow H} = ann^{\natural}_{\Omega \times TH}(gr(\pi_{\Omega}^{\sharp} \circ R^{-1}))$. Therefore, $\Pi_{\Omega \Rightarrow H} \Rightarrow \Pi_A$ is a Lie subalgebroid of $T\Omega \times T\Omega \Rightarrow TA \times TA$, as desired. The converse is proved simply by reversing the preceding arguments.
\end{proof}

The following Proposition will be useful.

\begin{proposicao}\label{La-Diracdoublealgbd}
    Given an LA-Dirac structure $L$ in an LA-Courant algebroid $\mathbb{A}$ as follows
    $$
    \begin{tikzcd}[
		column sep={2em,between origins},
		row sep={1.3em,between origins},]
		L & \Longrightarrow & W  \\
		\Downarrow & & \Downarrow \\
		A & \Longrightarrow & M
    \end{tikzcd}
    \quad \subseteq \quad
    \begin{tikzcd}[
		column sep={2em,between origins},
		row sep={1.3em,between origins},]
		\mathbb{A} & \Longrightarrow & H  \\
		\downarrowtail & & \downarrow \\
		A & \Longrightarrow & M,
	\end{tikzcd}
    $$
    the induced structure on $L$ makes it into a double Lie algebroid.
\end{proposicao}
\begin{proof}
    Since $L \rightarrow A$ is a Dirac structure on the Courant algebroid $\mathbb{A} \rightarrowtail A$, it is a Lie algebroid. By Proposition \ref{proprelationsubalgbd},  it suffices to show that $\Pi_{L \Rightarrow W} \Rightarrow \Pi_{A}$ is a Lie subalgebroid of $TL \times TL \Rightarrow TA \times TA$, which follows directly from the fact that $\Pi_{\mathbb{A}}$ is a Courant relation along $\Pi_A$.
\end{proof}

\subsection{Manin triples for double Lie bialgebroids}

The following notation will be useful in what follows. Given a Lie algebroid $A\Rightarrow M$ with anchor $\rho_A$ and bracket $[\cdot \,,\cdot]_A$, we denote by $A^{op}$ the \textit{opposite Lie algebroid}: the vector bundle $A\to M$ with anchor $-\rho_A$ and bracket $-[\cdot \,,\cdot]_A$. Note that for Manin triples $(\mathbb{E}, A, B)$ and $(\overline{\mathbb{E}}, A, B)$, the Lie algebroid structures on $A^*$, arising from its identification with $B$ through the respective pairings, are opposites of each other. If $\Omega$ is a double Lie algebroid as in \eqref{doublealgebroid}, we denote by $\Omega^{op_A}$ the double Lie algebroid given the same underlying double vector bundle of $\Omega$ with the Lie algebroid structures $\Omega^{op}$ over $A$ and the original one over $H$. 

\vspace{1mm}

Given a double Lie bialgebroid $(\Omega, \Omega^{*_A})$ as in \eqref{doublebialgebroid}, we know that $\Omega \oplus_A \Omega^{*_A} \rightarrowtail A$ is a Courant algebroid, and then its tangent lift $T(\Omega \oplus_A \Omega^*) \rightarrowtail TA$ is also a Courant algebroid. Therefore, we can consider the following product Courant algebroid
$$
T(\Omega \oplus_A \Omega^*) \times \overline{T(\Omega \oplus_A \Omega^*)}.
$$
Note that $(\mathbb{E}, L_1, L_2)$ is a Manin triple, where $\mathbb{E} = T(\Omega \oplus_A \Omega^*) \times \overline{T(\Omega \oplus_A \Omega^*)}$, $L_1 = T(\Omega \oplus 0) \times T(\Omega \oplus 0) \simeq T\Omega \times T\Omega$ and $L_2 = T(0 \oplus \Omega^*) \times T(0 \oplus \Omega^*) \simeq T\Omega^* \times T\Omega^*$. Furthermore, the identification between $\mathbb{E}$ and the Drinfeld double $L_1 \oplus L_1^*$, induced by the pairing of $\mathbb{E}$, is given by the map \eqref{eq:Phi}:
\begin{eqnarray}
\Phi : T(\Omega \oplus_A \Omega^*) \times \overline{T(\Omega \oplus_A \Omega^*)} & \longrightarrow & (T\Omega \times T\Omega) \oplus (T\Omega^* \times (T\Omega^*)^{op_{TA}}) \nonumber \\
(V, W) & \longmapsto & (dp_\Omega(V), dp_\Omega(W)) \oplus (dp_{\Omega^*}(V), - dp_{\Omega^*}(W)). \nonumber
\end{eqnarray}
Note also that the map \eqref{eq:phi} defines the following Lie algebroid morphism:
\begin{eqnarray}
\varphi: T\Omega^* \times T\Omega^* &\longrightarrow& T\Omega^* \times (T\Omega^*)^{op_{TA}} \nonumber \\
(\xi, \eta) &\longmapsto& (\xi, -\eta). \nonumber
\end{eqnarray}

\vspace{3mm}

Finally, we can prove our main Theorem.
\begin{teorema} \label{teorema2}
	{\em The Drinfeld double of a double Lie bialgebroid $(\Omega, \Omega^{*_A} )$ is an LA-Courant algebroid, so that $(\Omega \oplus \Omega^{*_A} , \Omega, \Omega^{*_A} )$ is an LA-Manin triple. Conversely, given an LA-Manin triple $(\mathbb{A}, L_1, L_2)$, $(L_1, L_2)$ is a double Lie bialgebroid, where $L_2$ is identified with $L_1^*$ via the pairing of $\mathbb{A}$.}
\end{teorema}
\begin{proof}
	Given a double Lie bialgebroid $(\Omega, \Omega^{*_A})$, its Drinfeld double is depicted as follows
$$
\begin{tikzcd}[
		column sep={2em,between origins},
		row sep={1.3em,between origins},]
		\,\,\,\, \Omega \oplus_A \Omega^{*_A} \,\,\,\,\,\,\,\,\,\,\,\,\,\,\,\,\,\,\, & \Longrightarrow & \,\,\,\,\,\,\,\,\,\,\, H \oplus C^*  \\
		\downarrowtail \,\,\,\,\,\,\,\,\,\,\,\,\,\,\,\,\,\,\, & & \,\,\,\,\,\,\,\,\, \downarrow \\
		A \,\,\,\,\,\,\,\,\,\,\,\,\,\,\,\,\,\,\, & \Longrightarrow & \,\,\,\,\,\,\,\,\, M,
\end{tikzcd}
$$
where $\Omega \oplus_A \Omega^{*_A} \rightarrowtail A$ is a Courant algebroid. Since $\Omega \times \Omega^{*_A}$ is a VB-algebroid and using the diagonal maps $A \rightarrow A \times A$ and $M \rightarrow M \times M$, by Proposition \ref{drinfelddoubleevbalgebroid} we conclude that $\Omega \oplus_A \Omega^{*_A}$ is a VB-algebroid. By considering $\Omega \oplus_A \Omega^{*_A}$ the horizontal VB-groupoid where $\mathtt{s} = \mathtt{t} = (q_H^{\Omega} \times q_{C^*}^{\Omega^{*_A}})\vert_{\Omega\oplus _{A} \Omega^{*_{A}}}$ and the multiplication is the sum on $\Omega \oplus_A \Omega^{*_A}$ over $H \oplus C^*$, by \cite[Theorem 4.6]{mancur1} this is a CA-groupoid, and therefore $\Omega \oplus_A \Omega^{*_A}$ is a VB-Courant algebroid. It remains to show that $\Pi_{\Omega \oplus_A \Omega^{*_A}}$ is a Courant relation. By Proposition \ref{prop:12} (b), and since $\Phi$ is a Courant algebroid morphism, showing that $\Pi_{\Omega \oplus_A \Omega^{*_A}}$ is Dirac on $T(\Omega \oplus_A \Omega^*) \times \overline{T(\Omega \oplus_A \Omega^*)}$ is equivalent to showing that $\Pi_\Omega \oplus_{\Pi_A} ann_{TA \times TA}(\Pi_\Omega)$ is Dirac on $(T\Omega \times T\Omega) \oplus (T\Omega^* \times (T\Omega^*)^{op_{TA}})$.
By Proposition \ref{proprelationsubalgbd} and Proposition \ref{proprelationdual} we know that $\Pi_{\Omega} \Rightarrow \Pi_A \text{\quad and \quad} \Pi_{\Omega^{*_A}} \simeq ann_{TA \times TA}^{\natural}(\Pi_\Omega) \Rightarrow \Pi_A$ are Lie subalgebroids of $T\Omega \times T\Omega \Rightarrow TA \times TA$ and $T\Omega^{*_A} \times T\Omega^{*_A} \Rightarrow TA \times TA$, respectively. Then, by Proposition \ref{prop:12} (a), and since $\varphi$ is a Lie algebroid morphism, $ann_{TA \times TA}(\Pi_\Omega)$ is a Lie subalgebroid of $T\Omega^* \times (T\Omega^*)^{op_{TA}}$. Therefore, it is immediate from Proposition \ref{proposicao} that $\Pi_{\Omega \oplus_A \Omega^{*_A}}$ is a Courant relation.

\vspace{1mm}

Conversely, consider $(L_1, L_2)$ two transverse LA-Dirac structures on an LA-Courant algebroid. By Proposition \ref{La-Diracdoublealgbd}, $L_1$ and $L_2$ are double Lie algebroids. By the theory of Manin triples for Lie bialgebroids \cite{liu1997manin}, the identification \( L_2^* \simeq L_1 \) makes the pair \( (L_1, L_2) \) into a Lie bialgebroid. Moreover, by Proposition~\ref{algbmorf}, under this identification, \( L_1 \) and \( L_2 \) are in duality as horizontal VB-algebroids. Therefore, \( (L_1, L_2) \) is a double Lie bialgebroid.
\end{proof}

\vspace{1mm}

\begin{exemplo}\label{exdoubletg}
    Let $(A \Rightarrow M, A^* \Rightarrow M)$ be a Lie bialgebroid and consider the double Lie bialgebroid $(TA, TA^*)$, as below. By Theorem \ref{teorema2}, its Drinfeld double is an LA-Courant algebroid as follows:
     $$
	\begin{tikzcd}[
		column sep={2em,between origins},
		row sep={1.4em,between origins},]
		TA \,\,\,\, & \Longrightarrow & A \\
		\Downarrow \,\,\,\, & & \Downarrow \\
		TM \,\,\,\, & \Longrightarrow & M
	\end{tikzcd}
	\, \oplus \,
	\begin{tikzcd}[
		column sep={2em,between origins},
		row sep={1.4em,between origins},]
		TA^* \,\,\,\, & \Longrightarrow & \,\, A^*  \\
		\Downarrow \,\,\,\, & &  \Downarrow \\
		TM \,\,\,\, & \Longrightarrow &  M
	\end{tikzcd}
	\,\,\, = \,\,\,
	\begin{tikzcd}[
		column sep={2em,between origins},
		row sep={1.4em,between origins},]
		TA \oplus TA^* \,\,\,\,\,\,\,\,\,\,\,\,\,\,\,\,\,\,\, & \Longrightarrow & \,\,\,\,\,\,\,\,\,\,\,\, A \oplus A^* \\
		\downarrowtail \,\,\,\,\,\,\,\,\,\,\,\,\,\,\,\,\,\,\, & & \,\,\,\,\,\,\,\,\,\, \downarrow \\
		TM \,\,\,\,\,\,\,\,\,\,\,\,\,\,\,\,\,\,\, &  \Longrightarrow & \,\,\,\,\,\,\,\,\,\, M,
	\end{tikzcd}
    $$
where the Courant algebroid structure on $TA \oplus TA^* \rightarrowtail TM$ is isomorphic to the tangent lift of the Courant algebroid $A \oplus A^* \rightarrowtail M$ (see \cite{boumaiza2009relevement}).
\end{exemplo}

\begin{exemplo}
Given a Poisson algebroid $(A \Rightarrow M, \pi)$, the tangent and cotangent bundles are double Lie algebroids in duality, and $(TA \Rightarrow A, T^*A \Rightarrow A)$ form a Lie bialgebroid. Then, its Drinfeld double is an LA-Courant algebroid:
$$
	\begin{tikzcd}[
		column sep={2em,between origins},
		row sep={1.3em,between origins},]
		TA \,\, & \Longrightarrow & \,\, TM \\
		\Downarrow \,\, & & \,\, \Downarrow \\
		A \,\, & \Longrightarrow & \,\, M
	\end{tikzcd}
	\, \oplus \,
	\begin{tikzcd}[
		column sep={2em,between origins},
		row sep={1.3em,between origins},]
		T^*A\,\,\,\, &\Longrightarrow & \,\, A^*  \\
		\Downarrow \,\,\,\, & & \Downarrow \\
		A \,\,\,\, & \Longrightarrow & M
	\end{tikzcd}
    \,\,\, = \,\,\,
	\begin{tikzcd}[
			column sep={2em,between origins},
			row sep={1.3em,between origins},]
			TA \oplus T^*A \,\,\,\,\,\,\,\,\,\,\,\,\,\,\,\,\,\,\,\,\, & \Longrightarrow & \,\,\,\,\,\,\,\,\,\,\,\,\,\,\,\,\,\, TM \oplus A^*  \\
			\downarrowtail \,\,\,\,\,\,\,\,\,\,\,\,\,\,\,\,\,\,\,\,\, & & \,\,\,\,\,\,\,\,\,\,\,\,\,\,\,\,\,\, \downarrow \\
			A \,\,\,\,\,\,\,\,\,\,\,\,\,\,\,\,\,\,\,\,\, & \Longrightarrow & \,\,\,\,\,\,\,\,\,\,\,\,\,\,\,\,\,\, M. \nonumber
	\end{tikzcd}
$$
 The Courant algebroid structure above is the standard one (Example \ref{standardCA}). See \cite{ten2019applications} for a detailed description of such an example in the language of graded manifolds.
\end{exemplo}

\begin{exemplo} (\cite[Example 3.3.3]{jotz2020courant})
    Given a double Lie algebroid $\Omega$, as in \eqref{doublealgebroid}, with core $C$, its dual $\Omega^{*_A}$ is also a double Lie algebroid, where the Lie algebroid structure on $\Omega^{*_A} \Rightarrow A$ is the trivial one, with $(\Omega \Rightarrow A, \Omega^{*_A} \Rightarrow A)$ a Lie bialgebroid. Therefore, $(\Omega, \Omega^{*_A})$ is a double Lie bialgebroid, and by Theorem \ref{teorema2}, its Drinfeld double is an LA-Courant algebroid as follows:
    $$
	\begin{tikzcd}[
		column sep={2em,between origins},
		row sep={1.4em,between origins},]
		\Omega \,\,\,\, & \Longrightarrow & H \\
		\Downarrow \,\,\,\, & & \Downarrow \\
		A \,\,\,\, & \Longrightarrow & M
	\end{tikzcd}
	\, \oplus \,
	\begin{tikzcd}[
		column sep={2em,between origins},
		row sep={1.4em,between origins},]
		\Omega^{*_A} \,\,\,\, & \Longrightarrow & \,\, C^*  \\
		\Downarrow \,\,\,\, & &  \Downarrow \\
		A \,\,\,\, & \Longrightarrow &  M
	\end{tikzcd}
	\,\,\, = \,\,\,
	\begin{tikzcd}[
		column sep={2em,between origins},
		row sep={1.4em,between origins},]
		\Omega \oplus \Omega^{*_A} \,\,\,\,\,\,\,\,\,\,\,\,\,\,\,\,\,\,\, & \Longrightarrow & \,\,\,\,\,\,\,\,\,\,\,\, H \oplus C^* \\
		\downarrowtail \,\,\,\,\,\,\,\,\,\,\,\,\,\,\,\,\,\,\, & & \,\,\,\,\,\,\,\,\,\, \downarrow \\
		A \,\,\,\,\,\,\,\,\,\,\,\,\,\,\,\,\,\,\, &  \Longrightarrow & \,\,\,\,\,\,\,\,\,\, M.
	\end{tikzcd}
    $$
\end{exemplo}

\section{Lie theory for Courant double structures given by Drinfeld doubles}\label{section5}

In this section, we highlight how our results provide a differentiation/integration result between particular cases of double structures involving Courant algebroids. 

\vspace{1mm}

A \textbf{Lie bialgebroid groupoid} \cite{bursztyn2021poisson} is a pair $(\Gamma, \Gamma^{*_G})$ of LA-groupoids \cite{mackenzie1992double} which are in duality as VB-groupoids,
\begin{eqnarray}\label{liebialgebroidgroupoid}
\begin{tikzcd}[
column sep={2em,between origins},
row sep={1.3em,between origins},]
\Gamma & \rightrightarrows & H  \\
\Downarrow & & \Downarrow \\
G & \rightrightarrows & M,
\end{tikzcd}
\,\,\,\,\,\,\,\,\,\,\,\,\,\,\,
\begin{tikzcd}[
column sep={2em,between origins},
row sep={1.3em,between origins},]
\Gamma^{*_G} & \rightrightarrows & C^*  \\
\Downarrow & & \Downarrow \\
G & \rightrightarrows & M,
\end{tikzcd}
\end{eqnarray}
and whose Lie algebroid structures $(\Gamma \Rightarrow G, \Gamma^{*_G} \Rightarrow G)$ form a Lie bialgebroid. In \cite[\S 5]{bursztyn2021poisson} it was proven that if $\Gamma$ is a source-connected LA-groupoid, then $(\Gamma, \Gamma^{*_G})$ is a Lie bialgebroid groupoid if and only if $(\Omega, \Omega^{*_A})$ is a double Lie bialgebroid, where $\Omega$ and $\Omega^{*_A}$ are the double Lie algebroids corresponding to the differentiation of the LA-groupoids $\Gamma$ and $\Gamma^{*_G}$. 

\vspace{1mm}

Given a multiplicative Manin triple $\mathcal{G} := \Gamma \oplus_G \Gamma^{*_G}$, the author proved in \cite[Theorem 4.6]{mancur1} that it corresponds to a Lie bialgebroid groupoid $(\Gamma, \Gamma^{*_G})$. By applying the Lie functor to the horizontal groupoid structures on the Lie bialgebroid groupoid, it automatically follows from the Lie theory of Poisson double structures \cite{bursztyn2021poisson} that its differentiation
$$
\begin{tikzcd}[
		column sep={2em,between origins},
		row sep={1.3em,between origins},]
		\Gamma & \rightrightarrows & H  \\
		\Downarrow & & \Downarrow \\
		G & \rightrightarrows & M
\end{tikzcd}
\, \oplus_G \,
\begin{tikzcd}[
		column sep={2em,between origins},
		row sep={1.3em,between origins},]
		\Gamma^{*_G} \,\,\, & \rightrightarrows & \,\, C^*  \\
		\Downarrow \,\,\, & & \Downarrow \\
		G \,\,\, & \rightrightarrows & M
\end{tikzcd}
\quad \overset{\text{Lie}}{\longmapsto} \quad
\begin{tikzcd}[
		column sep={2em,between origins},
		row sep={1.3em,between origins},]
		\Omega & \Longrightarrow & H  \\
		\Downarrow & & \Downarrow \\
		A & \Longrightarrow & M
\end{tikzcd}
\, \oplus_A \,
\begin{tikzcd}[
		column sep={2em,between origins},
		row sep={1.3em,between origins},]
		\Omega^{*_A} \,\,\, & \Longrightarrow & \,\, C^*  \\
		\Downarrow \,\,\, & & \Downarrow \\
		A \,\,\, & \Longrightarrow & M
\end{tikzcd}
$$
is a double Lie bialgebroid $(\Omega, \Omega^{*_A})$, and by Theorem \ref{teorema2}, it gives rise to an LA-Manin triple. The same happens with the integration: if $\mathbb{A} := \Omega \oplus_A \Omega^{*_A}$ is an LA-Manin triple whose top Lie algebroid $\Omega \Rightarrow H$ is integrable, then its source-simply-connected integration is a CA-groupoid given by the Drinfeld double of Lie bialgebroids groupoids. Therefore, CA-groupoids and LA-Courant algebroids given by Drinfeld doubles are related via differentiation/integration. 

\vspace{1mm}

The following diagram illustrates how to describe the Lie theory relating Courant double structures, bringing together our main result with \cite[Theorem 4.6]{mancur1} and the Lie theory of Poisson double structures \cite{bursztyn2021poisson}:
$$
\xymatrix@C=40pt{
\text{multiplicative Manin triples \,} \ar@{<->}[d]_{} \ar@{<-->}[r]^{Lie}&\text{\, LA-Manin triples} \ar@{<->}[d]^{}\\
\text{Lie bialgebroid groupoids \,}
\ar@{<->}[r]_{Lie}&\text{\, double Lie bialgebroids} }
$$

\vspace{2mm}

In what follows, we use the fact that, for any Lie bialgebroid groupoid $(\Gamma, \Gamma^{*_G})$, the groupoid $\Gamma$ is source-simply-connected if and only if its base groupoid $G$ is (see \cite[Remark 3.1.1]{bursztyn2016vector}). We summarize the discussion above in the following result. 

\begin{teorema}\label{teodifintdoubles}
    Let $(\Gamma, \Gamma^{*_G})$ be a Lie bialgebroid groupoid as in \eqref{liebialgebroidgroupoid}, with $G$ source-simply-connected, and let $(\Omega, \Omega^{*_A}) = \mathrm{Lie}(\Gamma, \Gamma^{*_G})$ denote its associated double Lie bialgebroid. Then the Lie functor establishes a one-to-one correspondence between multiplicative Manin triples $(\mathcal{G}, \Gamma, \Gamma^{*_G})$ and LA-Manin triples $(\mathbb{A}, \Omega, \Omega^{*_A})$.
\end{teorema}

\begin{exemplo}\label{expairgrpdinf}
    Let $(A \Rightarrow M, A^* \Rightarrow M)$ be a Lie bialgebroid with $M$ simply-connected. We can consider the Lie bialgebroid groupoid $(A \times A, A^* \times A^*)$ by equipping it with the structure of the pair groupoid. By \cite[Theorem 4.16]{mancur1}, its Drinfeld double is a CA-groupoid as follows:
    $$
	\begin{tikzcd}[
		column sep={2em,between origins},
		row sep={1.3em,between origins},]
		A \times A \,\,\,\,\,\,\,\,\,\,\, & \rightrightarrows & A \\
		\Downarrow \,\,\,\,\,\,\,\,\,\,\, & & \Downarrow \\
		M {\times} M \,\,\,\,\,\,\,\,\,\,\, & \rightrightarrows & M
	\end{tikzcd}
	\,\,\, \oplus \,\,\,
	\begin{tikzcd}[
		column sep={2em,between origins},
		row sep={1.3em,between origins},]
		A^* \times A^* \,\,\,\,\,\,\,\,\,\,\, & \rightrightarrows & A^*  \\
		\Downarrow \,\,\,\,\,\,\,\,\,\,\, & &  \Downarrow \\
		M {\times} M \,\,\,\,\,\,\,\,\,\,\, & \rightrightarrows &  M
	\end{tikzcd}
    \quad = \quad
    \begin{tikzcd}[
		column sep={2em,between origins},
		row sep={1.3em,between origins},]
		\mathbb{E} \times \mathbb{E} \,\,\,\,\,\,\,\,\,\,\, & \rightrightarrows & \,\,\,\,\,\,\,\,\,\, A \oplus A^* \\
		\downarrowtail \,\,\,\,\,\,\,\,\,\,\, & & \,\,\,\,\,\,\,\,\,\, \downarrow \\
		M {\times} M \,\,\,\,\,\,\,\,\,\,\, &  \rightrightarrows & \,\,\,\,\,\,\,\,\,\, M,
	\end{tikzcd}
	$$
    where $\mathbb{E}$ is the Courant algebroid $A \oplus A^* \rightarrowtail M$. Since the Lie algebroid of the pair groupoid $M \times M \rightrightarrows M$ is the tangent Lie algebroid $TM \Rightarrow M$ (\cite[\S3.5]{mackenzie2005general}) and the pair groupoid $M \times M \rightrightarrows M$ integrates $TM \Rightarrow M$ when $M$ is simply-connected, it follows from the above discussion that the previously described CA-groupoid and the LA-Courant algebroid in Example \ref{exdoubletg} are related via differentiation and integration.
\end{exemplo}

\appendix

\section{Double and triple vector bundles}\label{apendicetriplevb}

\subsection{Double vector bundles}\label{sectiondvb}

In this section, we recall the definition and some useful properties of double vector bundles. For more details see \cite{mackenzie2005general}.

\vspace{1mm}

\begin{definicao}\label{defdvb}
	A \textbf{double vector bundle} $(D, B, A, M)$ is a commutative diagram
	\begin{eqnarray}\label{doublevb}
	\begin{tikzcd}[
		column sep={2em,between origins},
		row sep={1.3em,between origins},]
		D & \longrightarrow & B  \\
		\downarrow & & \downarrow \\
		A & \longrightarrow & M,
	\end{tikzcd}
	\end{eqnarray}
	in which every arrow is a vector bundle and the two vector bundle structures on $D$ are compatible, in the sense that the structural maps (projection, zero section, fiberwise addition and multiplication by scalars) are vector bundle maps. Equivalently, $D$ as above is a double vector bundle if the horizontal and vertical scalar multiplications commute (see \cite{grabowski2009higher, bursztyn2016vector}).
\end{definicao}

The \textbf{core} of a double vector bundle is the subbundle $C \rightarrow M$ given by the kernel of the double vector bundle map $D \rightarrow B \oplus A$ given by the projections on the side bundles,
$$
C \hookrightarrow D \rightarrow B \oplus A
$$
where $C$ is regarded as a double vector bundle with trivial sides, and $B \oplus A$ is a double vector bundle with trivial core. This sequence splits (non-canonically), giving an isomorphism $D \simeq B \oplus A \oplus C$, called \textbf{splitting}, inducing the identity on the sides and core (see \cite{gracia2010lie}). 

\begin{exemplo}\label{tangcotangdvb}
	Given a vector bundle $E \rightarrow M$, its tangent and cotangent bundles are double vector bundles:
	$$
	\begin{tikzcd}[
		column sep={2em,between origins},
		row sep={1.3em,between origins},]
		TE \,\,\,\,\, & \longrightarrow & TM  \\
		\downarrow & & \downarrow \\
		E & \longrightarrow & M,
	\end{tikzcd}
\,\,\,\,\,\,\,\,\,\,
	\begin{tikzcd}[
		column sep={2em,between origins},
		row sep={1.3em,between origins},]
		T^*E & \longrightarrow & E^*  \\
		\downarrow & & \downarrow \\
		E & \longrightarrow & M.
	\end{tikzcd}
	$$
	The core of $TE$ is identified with $E$, and $T^*E$ has core $T^*M$.
\end{exemplo}

For a double vector bundle  $D$ as in \eqref{doublevb}, we can consider its vertical dual $D^{*_A} = D^{*_y} \rightarrow A$ and its horizontal dual $D^{*_B} = D^{*_x} \rightarrow B$, and both are again double vector bundles (with core bundles indicated in the middle of the following diagrams):
$$
\begin{tikzcd}[
		column sep={2em,between origins},
		row sep={1.3em,between origins},]
		D^{*_A} & \longrightarrow & C^*  \\
		\downarrow & {\scriptstyle B^*} & \downarrow \\
		A & \longrightarrow & M,
\end{tikzcd}
\qquad
\begin{tikzcd}[
		column sep={2em,between origins},
		row sep={1.3em,between origins},]
		D^{*_B} & \longrightarrow & B  \\
		\downarrow & {\scriptstyle A^*} & \downarrow \\
		C^* & \longrightarrow & M.
\end{tikzcd}
$$
To see more details about duals of double vector bundles, see \cite[\S 9.2]{mackenzie2005general}. The tangent and cotangent double vector bundles in Example \ref{tangcotangdvb} are related by vertical duality.

\vspace{1mm}

In \cite[\S 9.2]{mackenzie2005general} it was shown that there exists a natural pairing between $D^{*_B}$ and $D^{*_A}$, given by
\begin{eqnarray}
    \vertbold \phi, \psi \vertbold = \langle \phi , d \rangle_B - \langle \psi , d \rangle_A, \nonumber
\end{eqnarray}
where $\phi \in D^{*_B}$, $\psi \in D^{*_A}$ satisfy $q^{D^{*_B}}_{C^*}(\phi) = q^{D^{*_A}}_{C^*}(\psi)$, $d$ is any element of $D$ with $q^D_B(d) = q^{D^{*_B}}_B(\phi)$ and $q^D_A (d) = q^{D^{*_A}}_A(\psi)$, and $\langle \,\, , \, \rangle_B$, $\langle \,\, , \, \rangle_A$ denote the natural pairings between $D$ and $D^{*_B}$ or $D^{*_A}$, respectively. This pairing induces the following isomorphisms of double vector bundles:
\begin{eqnarray}
        Z_B : D^{*_B} \to D^{*_A*_{C^*}}, \quad \langle Z_B(\phi), \psi \rangle_{C^*} = \vertbold \phi, \psi \vertbold, \nonumber \\
        Z_A : D^{*_A} \to D^{*_B*_{C^*}}, \quad \langle Z_A(\psi), \phi \rangle_{C^*} = \vertbold \phi, \psi \vertbold. \nonumber
\end{eqnarray}

\vspace{1mm}

Given a vector bundle $E \to M$, by using the above isomorphisms for the double vector bundle $TE$ and the canonical pairing between $E$ and $E^*$, Mackenzie defined a map
$$
R : T^*(E^*) \to T^*E
$$
called the \textbf{reversal isomorphism}, and proved that it is an isomorphism of double vector bundles preserving the side bundles and inducing $-id$ on the core \cite[\S 9.5]{mackenzie2005general}. Further, in \cite[Theorem 9.5.1]{mackenzie2005general} it was shown that, for all $U \in TE, \mu \in TE^*, \beta \in T^*E^*$,
\[
\langle\!\langle \mu , U \rangle\!\rangle_{TM} = \langle R(\beta), U \rangle_E + \langle \beta, \mu \rangle_{E^*},
\]
where $\langle\!\langle \,\, , \, \rangle\!\rangle_{TM}$ denotes the tangent lift of the natural pairing between $E$ and $E^*$, and $\langle \,\, , \, \rangle_E$, $\langle \,\, , \, \rangle_{E^*}$ denote the natural pairings between the tangent and cotangent bundles of $E$ and $E^*$, respectively.

\subsection{Triple vector bundles}

The dual theory of triple vector bundles is essential for defining LA-Courant algebroids, and we utilize it multiple times in this paper. In this section, we recall the definition of triple vector bundles, some dual properties and important examples. For more details, see \cite{gracia2009duality} and \cite{mackenzie2005duality}.

\begin{definicao}
	A \textbf{triple vector bundle} is a cube as follows, such that each edge is a vector bundle and each face is a double vector bundle:
\begin{eqnarray}\label{triplevb}
    \begin{minipage}{.25\textwidth}
	\begin{tikzcd}[back line/.style={densely dotted}, row sep=0.6em, column sep=0.6em]
		& E \ar{dl}[swap]{} \ar{rr} \ar[back line]{dd}
		& & D_1 \ar{dd}{} \ar{dl}[swap,sloped,near start]{} \\
		D_2 \ar[crossing over]{rr}[near start]{} \ar{dd}[swap]{} 
		& & V_3 \\
		& D_3 \ar[back line]{rr} \ar[back line]{dl} 
		& & V_2. \ar{dl} \\
		V_1 \ar{rr} & & M \ar[crossing over, leftarrow]{uu}
	\end{tikzcd}
    \end{minipage}
\end{eqnarray}
\end{definicao}

\vspace{1mm}

We will denote $C_1$, $C_2$ and $C_3$ the cores of the double vector bundles $(E, D_2, D_3, V_1)$, $(E, D_1, D_3, V_2)$ and $(E, D_1, D_2, V_3)$, respectively, and by $K_1$, $K_2$ and $K_3$ the cores of the double vector bundles $(D_1, V_3, V_2, M)$, $(D_2, V_3, V_1, M)$ and $(D_3, V_2, V_1, M)$, respectively. Then we have three core double vector bundles, given by $(C_1, K_1, V_1, M)$, $(C_2, K_2, V_2, M)$ and $(C_3, V_3, K_3, M)$. These last three double vector bundle have the same core, denoted by $C$ and called the \textbf{ultracore}. 

\begin{exemplo}
	Given a double vector bundle $(D, B, A, M)$, the tangent bundle $TD$ is a triple vector bundle:
    $$
    \begin{minipage}{.25\textwidth}
	\begin{tikzcd}[back line/.style={densely dotted}, row sep=0.6em, column sep=0.6em]
		& TD \ar{dl}[swap]{} \ar{rr} \ar[back line]{dd}
		& & TB \ar{dd}{} \ar{dl}[swap,sloped,near start]{} \\
		D \ar[crossing over]{rr}[near start]{} \ar{dd}[swap]{} 
		& & B \\
		& TA \ar[back line]{rr} \ar[back line]{dl} 
		& & TM. \ar{dl} \\
		A \ar{rr} & & M \ar[crossing over, leftarrow]{uu}
	\end{tikzcd}
    \end{minipage}
    $$
The ultracore of $TD$ is the core of $D$.
\end{exemplo}

Since a triple vector bundle $E$ as in \eqref{triplevb} carries three vector bundle structures, we can dualize $E$ in three different directions. Each of these duals is a triple vector bundle, which we denote by $E^{*_{D_1}}$, $E^{*_{D_2}}$, and $E^{*_{D_3}}$, indicating the vector bundle on which we are considering the dual. The dual of a general vector bundle is given by
$$
\begin{minipage}{.25\textwidth}
	\begin{tikzcd}[back line/.style={densely dotted}, row sep=0.6em, column sep=0.6em]
		& E^{*_{D_1}} \ar{dl}[swap]{} \ar{rr} \ar[back line]{dd}
		& & D_1 \ar{dd}{} \ar{dl}[swap,sloped,near start]{} \\
		C_3^{*_{V_3}} \ar[crossing over]{rr}[near start]{} \ar{dd}[swap]{} 
		& & V_3 \\
		& C_2^{*_{V_2}} \ar[back line]{rr} \ar[back line]{dl} 
		& & V_2, \ar{dl} \\
		C^* \ar{rr} & & M \ar[crossing over, leftarrow]{uu}
	\end{tikzcd}
\end{minipage}
$$
where the ultracore is $V_1^*$, and the following side double vector bundles have the indicated cores:
\begin{eqnarray}
	\begin{tikzcd}[
		column sep={2em,between origins},
		row sep={1.3em,between origins},]
		E^{*_{D_1}} & \rightarrow & C_3^{*_{V_3}}  \\
		\downarrow & {\scriptstyle C_1^{*_{V_1}}} & \downarrow \\
		C_2^{*_{V_2}} & \rightarrow & C^*,
	\end{tikzcd}
   \qquad
    \begin{tikzcd}[
		column sep={2em,between origins},
		row sep={1.3em,between origins},]
		C_3^{*_{V_3}} & \rightarrow & V_3  \\
		\downarrow & {\scriptstyle K_3^*} & \downarrow \\
		C^* & \rightarrow & M,
	\end{tikzcd}
  \qquad
    \begin{tikzcd}[
		column sep={2em,between origins},
		row sep={1.3em,between origins},]
		C_2^{*_{V_2}} & \rightarrow & V_2  \\
		\downarrow & {\scriptstyle K_2^*} & \downarrow \\
		C^* & \rightarrow & M.
	\end{tikzcd} \nonumber
\end{eqnarray}

\begin{exemplo}
	For the tangent prolongation of a double vector bundle $(D, A, B, M)$ with core $C$, we can consider the following duals, given, respectively, by the cotangent of $D$ and the dual of $TD \rightarrow TA$:
$$
\begin{minipage}{.23\textwidth}
	\begin{tikzcd}[back line/.style={densely dotted}, row sep=0.6em, column sep=0.6em]
		& T^*D \ar{dl}[swap]{} \ar{rr} \ar[back line]{dd}
		& & D^{*_B} \ar{dd}{} \ar{dl}[swap,sloped,near start]{} \\
		D \ar[crossing over]{rr}[near start]{} \ar{dd}[swap]{} 
		& & B \\
		& D^{*_A} \ar[back line]{rr} \ar[back line]{dl} 
		& & C^*, \ar{dl} \\
		A \ar{rr} & & M \ar[crossing over, leftarrow]{uu}
	\end{tikzcd}
\end{minipage}  
\qquad \qquad \qquad
\begin{minipage}{.32\textwidth}
	\begin{tikzcd}[back line/.style={densely dotted}, row sep=0.6em, column sep=0.6em]
		& (TD)^{*_{TA}} \ar{dl}[swap]{} \ar{rr} \ar[back line]{dd}
		& & T^*C \ar{dd}{} \ar{dl}[swap,sloped,near start]{} \\
		D^{*_A} \ar[crossing over]{rr}[near start]{} \ar{dd}[swap]{} 
		& & C^* \\
		& TA \ar[back line]{rr} \ar[back line]{dl} 
		& & TM. \ar{dl} \\
		A \ar{rr} & & M \ar[crossing over, leftarrow]{uu}
	\end{tikzcd}
\end{minipage}
$$
The duality between $A$ and $A^*$ induces a duality between $TA \rightarrow TM$ and $T(A^*) \rightarrow TM$, and induces an isomorphism $T(A^*) = T(A^{*_M}) \rightarrow (TA)^{*_{TM}}$ (see \cite[\S 9.3]{mackenzie2005general}). These combine into an isomorphism between the triple vector bundle $(TD)^{*_{TA}}$ and the tangent prolongation of $(D^{*_A}, A, C^*, M)$. Note that the ultracore of $T^*D$ is $T^*M$, then dualizing $T^*D$ along the $x$-axis (along $D^{*_B}$), we get
$$
    \begin{minipage}{.25\textwidth}
	\begin{tikzcd}[back line/.style={densely dotted}, row sep=0.6em, column sep=0.6em]
		& (T^*D)^{*_x} \ar{dl}[swap]{} \ar{rr} \ar[back line]{dd}
		& & D^{*_B} \ar{dd}{} \ar{dl}[swap,sloped,near start]{} \\
		TB \ar[crossing over]{rr}[near start]{} \ar{dd}[swap]{} 
		& & B \\
		& TC^* \ar[back line]{rr} \ar[back line]{dl} 
		& & C^*. \ar{dl} \\
		TM \ar{rr} & & M \ar[crossing over, leftarrow]{uu}
	\end{tikzcd}
    \end{minipage}
$$
\end{exemplo}

\vspace{2mm}

In \cite[Theorem 6.1]{mackenzie2005duality},  Mackenzie proved the following: the map $R^{-1}_{D \to B} : T^*D \rightarrow T^*(D^{*_B})$, arising from the vector bundle $D \rightarrow B$, is an isomorphism of triple vector bundles over $Z_A : D^{*_A} \rightarrow D^{*_B*_{C^*}}$. The following proposition relates the different reversal isomorphisms on a double vector bundle, and its proof is straightforward.

\begin{proposicao}\label{propreversalcommute}
The following diagram commutes
\[
\begin{tikzcd}[column sep=large, row sep=large]
T^*\Omega^{*_A*_{C^*}} \arrow[r,"R_{\Omega^{*_A} \to C^*}"] \arrow[d,"(dZ_H)^{*_{\Omega^{*_B}}}"'] & T^*\Omega^{*_A} \arrow[d,"-_{\Omega^{*_A}} R_{\Omega \to A}"] \\
T^*\Omega^{*_B} \arrow[r,"R_{\Omega \to B}"'] & T^*\Omega,
\end{tikzcd}
\]
where the maps $R_{\Omega \to B}, R_{\Omega \to A}$ and $R_{\Omega^{*_A} \to C^*}$ are the reversal isomorphisms arising from the indicated vector bundles, $dZ_H$ is the tangent lift of the isomorphism $Z_H : \Omega^{*_B} \to \Omega^{*_A*_{C^*}}$ defined on \S\ref{sectiondvb}, and $-_{\Omega^{*_A}}$ is the scalar multiplication by $-1$ of the triple vector bundle $T^*\Omega$ with respect to $\Omega^{*_A}$.
\end{proposicao}

 \bibliographystyle{acm}
 \bibliography{biblio2}

\begin{thebibliography}{10}

\bibitem{boumaiza2009relevement}
{\sc Boumaiza, M., and Zaalani, N.}
\newblock Rel{\`e}vement d'une alg{\'e}bro{\"\i}de de {C}ourant.
\newblock {\em Comptes Rendus. Math{\'e}matique 347}, 3-4 (2009), 177--182.

\bibitem{bursztyn2016vector}
{\sc Bursztyn, H., Cabrera, A., and del Hoyo, M.}
\newblock Vector bundles over {L}ie groupoids and algebroids.
\newblock {\em Advances in Mathematics 290\/} (2016), 163--207.

\bibitem{bursztyn2021poisson}
{\sc Bursztyn, H., Cabrera, A., and Hoyo, M.~d.}
\newblock Poisson double structures.
\newblock {\em Journal of Geometric Mechanics 14}, 2 (2022).

\bibitem{dorfman1993dirac}
{\sc Dorfman, I.}
\newblock Dirac structures and integrability of nonlinear evolution equations.
\newblock {\em (Nonlinear Science: Theory and Applications). Wiley \& Sons Ltd.\/} (1993).

\bibitem{drinfel1988quantum}
{\sc Drinfel'd, V.~G.}
\newblock Quantum groups.
\newblock {\em Journal of Soviet mathematics 41\/} (1988), 898--915.

\bibitem{drinfel1990hamiltonian}
{\sc Drinfel'd, V.~G.}
\newblock Hamiltonian structures on {L}ie groups, {L}ie bialgebras and the geometric meaning of the classical {Y}ang-{B}axter equations.
\newblock In {\em Yang-Baxter Equation In Integrable Systems}. World Scientific, 1990, pp.~222--225.

\bibitem{grabowski2009higher}
{\sc Grabowski, J., and Rotkiewicz, M.}
\newblock Higher vector bundles and multi-graded symplectic manifolds.
\newblock {\em Journal of Geometry and Physics 59}, 9 (2009), 1285--1305.

\bibitem{gracia2009duality}
{\sc Gracia-Saz, A., and Mackenzie, K. C.~H.}
\newblock Duality functors for triple vector bundles.
\newblock {\em Letters in Mathematical Physics 90\/} (2009), 175--200.

\bibitem{gracia2010lie}
{\sc Gracia-Saz, A., and Mehta, R.~A.}
\newblock Lie algebroid structures on double vector bundles and representation theory of {L}ie algebroids.
\newblock {\em Advances in Mathematics 223}, 4 (2010), 1236--1275.

\bibitem{jotz2020courant}
{\sc Jotz~Lean, M.}
\newblock On {LA}-courant algebroids and {P}oisson {L}ie 2-algebroids.
\newblock {\em Mathematical Physics, Analysis and Geometry 23}, 3 (2020), 31.

\bibitem{li2012courant}
{\sc Li-Bland, D.~S.}
\newblock {\em L{A}-{C}ourant algebroids and their applications}.
\newblock University of Toronto (Canada), 2012.

\bibitem{liu1997manin}
{\sc Liu, Z.-J., Weinstein, A., and Xu, P.}
\newblock Manin triples for {L}ie bialgebroids.
\newblock {\em Journal of Differential Geometry 45}, 3 (1997), 547--574.

\bibitem{mackenzie1992double}
{\sc Mackenzie, K.~C.}
\newblock Double {L}ie algebroids and second-order geometry {I}.
\newblock {\em Advances in Mathematics 94}, 2 (1992), 180--239.

\bibitem{mackenzie9808081double}
{\sc Mackenzie, K.~C.}
\newblock Double {L}ie algebroids and the double of a {L}ie bialgebroid.
\newblock {\em arXiv preprint math/9808081\/} (1998).

\bibitem{mackenzie1999symplectic}
{\sc Mackenzie, K.~C.}
\newblock On symplectic double groupoids and the duality of {P}oisson groupoids.
\newblock {\em International Journal of Mathematics 10}, 04 (1999), 435--456.

\bibitem{mackenzie2005duality}
{\sc Mackenzie, K.~C.}
\newblock Duality and triple structures.
\newblock In {\em The Breadth of Symplectic and Poisson Geometry: Festschrift in Honor of Alan Weinstein}. Springer, 2005, pp.~455--481.

\bibitem{mackenzie2005general}
{\sc Mackenzie, K.~C.}
\newblock {\em General theory of {L}ie groupoids and {L}ie algebroids}.
\newblock No.~213. Cambridge University Press, 2005.

\bibitem{mackenzie2006ehresmann}
{\sc Mackenzie, K.~C.}
\newblock Ehresmann doubles and {D}rinfel'd doubles for {L}ie algebroids and {L}ie bialgebroids.
\newblock {\em Walter de Gruyter GmbH \& Co. KG\/} (2011).

\bibitem{mackenzie1994lie}
{\sc Mackenzie, K.~C., and Xu, P.}
\newblock Lie bialgebroids and {P}oisson groupoids.
\newblock {\em Duke Mathematical Journal 73}, 2 (1994), 415--452.

\bibitem{mancur1}
{\sc Mançur, A.~C.}
\newblock Manin triples on multiplicative {C}ourant algebroids.
\newblock {\em arXiv:2505.01276\/} (2025).

\bibitem{mehta2009q}
{\sc Mehta, R.~A.}
\newblock Q-groupoids and their cohomology.
\newblock {\em Pacific journal of mathematics 242}, 2 (2009), 311--332.

\bibitem{cristianthesis}
{\sc Ortiz, C.}
\newblock {\em {M}ultiplicative {D}irac structures}.
\newblock Phd thesis, Instituto de Matemática Pura e Aplicada, 2009.

\bibitem{roytenberg1999courant}
{\sc Roytenberg, D.}
\newblock {\em Courant algebroids, derived brackets and even symplectic supermanifolds}.
\newblock University of California, Berkeley, 1999.

\bibitem{ten2019applications}
{\sc Ten, M.~C.}
\newblock {\em Applications of graded manifolds to Poisson geometry}.
\newblock PhD thesis, PhD thesis, IMPA, 2019.

\end{thebibliography}

\end{document}